\documentclass[journal]{IEEEtran}
\IEEEoverridecommandlockouts
\usepackage{amsmath, amsfonts, amsthm,amssymb, mathrsfs, setspace, graphicx, epsfig, caption, subcaption, url, float, framed, circuitikz, color,overpic,epstopdf,hyperref}

\newtheorem{theorem}{Theorem}
\newtheorem{example}{Example}
\newtheorem{definition}{Definition}
\newtheorem{lemma}{Lemma}

\makeatletter
\usepackage{cite}
\makeatother

\newcommand{\MM}{\mathcal{M}}  
\newcommand{\DD}{\mathrm{D}}   
\newcommand{\MC}{\overline{\mathcal{M}(\mathrm{D})}}  
\newcommand{\RR}{\mathbb{R}}  
\newcommand{\MB}{\partial \mathcal{M}(\mathrm{D})} 
\newcommand{\MMR}{\mathcal{\tilde{M}}}  
\newcommand{\DDR}{\mathrm{\tilde{D}}}  

\title{Model Boundary Approximation Method as a Unifying Framework for Balanced Truncation and Singular Perturbation Approximation}

\author{Philip E. Par\'{e}, David Grimsman, Alma T. Wilson,
Mark K. Transtrum and
Sean Warnick\thanks{
Philip E. Par\'{e} is at KTH Royal Institute of Technology and can be reached at {\tt philip.e.pare@gmail.com}.  David Grimsman is at the University of California at Santa Barbara and can be reached at {\tt davidgrimsman@gmail.com}. Alma T. Wilson, Mark K. Transtrum and Sean C. Warnick are at Brigham Young University, Provo, Utah, and can be reached at {\tt alba\_wilzod@yahoo.com}, {\tt mktranstrum@byu.edu}, and {\tt sean.warnick@gmail.com}, respectively.  This material is based on research sponsored by the Department of Homeland
Security (DHS) Science and Technology Directorate, Homeland Security Advanced Research Projects Agency (HSARPA), Cyber Security Division (DHS S\&T/HSARPA/CSD), LRBAA 12-07 via contract number HSHQDC-13-C-B0052.  All material in this paper represents the position of the authors and not necessarily that of DHS.}}
\begin{document}
\maketitle

\begin{abstract}
We show that two widely accepted model reduction techniques, Balanced Truncation and Balanced Singular Perturbation Approximation, can be derived as limiting approximations of a carefully constructed parameterization of Linear Time Invariant (LTI) systems by employing the Model Boundary Approximation Method (MBAM), a recent development in the Physics literature \cite{Transtrum2014}. 
This unifying framework of these popular model reduction techniques shows that Balanced Truncation and Balanced Singular Perturbation Approximation each correspond to a particular boundary point on a manifold, the ``model manifold,'' which is associated with the specific choice of model parameterization and initial condition, and is embedded in a sample space of measured outputs, which can be chosen arbitrarily, provided that the number of samples exceeds the number of parameters. 
We also show that MBAM provides a novel way to interpolate between Balanced Truncation and Balanced Singular Perturbation Approximation, by exploring the set of approximations on the boundary of the manifold between the elements that correspond to the two model reduction techniques; this allows for alternative approximations of a given system to be found that may be better under certain conditions. 
The work herein suggests similar types of approximations may be obtainable in topologically similar places (i.e. on certain boundaries) on the model manifold of nonlinear systems if analogous parameterizations can be achieved, therefore extending these widely accepted model reduction techniques to nonlinear systems.  
\end{abstract}

\section{Introduction}
\label{sec:intro}

Modern  systems theorists  are studying and engineering  systems that are larger and more complex than ever before \cite{strogatz2001exploring}. Typical examples of these complex systems include economic networks \cite{jackson1996strategic}, biological systems \cite{hartwell1999molecular,bhalla1999emergent,jeong2000large
}, and the Internet \cite{broder2000graph}.  Due to the unprecedented size of these  systems, simplified models are necessary to reason about them effectively \cite{antoulas2005approximation,benner2015survey}.  Specifically,  we detail four important motivations for building simplified approximations of large scale complicated systems, and then  provide an  overview  of how this  work relates  singular perturbation and balanced truncation methods.

\subsection{Motivation}

When attempting to learn a system from limited data, large models cannot be identified, making simplified models necessary.  First principles models typically have many parameters that must be tuned correctly for the model to reflect the behavior of a real system.  Examples are everywhere, from agronomy and  biochemical reaction networks to swarms of autonomous flying robots or power systems.  Using data to learn the correct values of parameters is the purview of system identification, and a rich theory has developed quantifying when data is informative enough to accurately estimate parameter values \cite{moore1983persistence,bitmead1984persistence,
shimkin1987persistency,katayama2006subspace}.  Typically, however, there is much less data than needed to learn all the parameters in a first-principles model, so simplifying the model to yield one with fewer parameters can help identify the system from data.  

Second, the need for simplified models arises when designing controllers for complex systems.  The complexity of an optimal controller often mirrors that of the system being controlled; therefore a complex system may often suggest the need for complicated controllers.  Nevertheless, when engineering such complicated systems is unreasonable, designing controllers for simplified approximations can lead to acceptable trade-offs between complexity and performance degradation.

Third, simplified models can be an important link between macro-scale and micro-scale models. Generative models, for example, often detail micro-scale phenomena, such as consumer-choice models or models of a single neuron or molecular organization, etc., and then hypothesize that macro-scale behavior, such as consumer demand or regions of the brain or material properties, etc., is the aggregation of a large number of micro-scale instances\cite{usanmaz2016first,franci2014modeling,epstein2006generative}. The complexity of modeling a macro-scale model composed of thousands or millions of micro-scale instances, however, can be unwieldy, and such models almost never exhibit behavioral complexity commensurate with the descriptive complexity of the model as an aggregate of  many micro-scale models. Thus, a systematic technique for  developing simplified descriptions of macro-scale models that resemble, in a principled way, the aggregate of micro-scale phenomena can be critical in such applications.

Fourth, understanding the resilience and vulnerability of large-scale critical infrastructures demands techniques for modeling the attack surface of complex cyber-physical-human systems.  The attack surface is typically a simplified model of the system that highlights the exposed variables and dynamics strongly affected by or observed from 
them \cite{chetty2014,cox2014,grimsman2016,teixeira2015secure,zhu2015game,amin2015game}.  Thus, the techniques discussed here offer the possibility of  modeling the attack surface of  large-scale cyber-physical-human systems and contribute to a science of system security.

\subsection{Overview}

Two  important model reduction techniques for linear systems include Balanced Truncation and singular perturbation methods.   Each of these approaches focuses on a particular aspect of the system to preserve.

Balanced Truncation (BT) was first proposed in \cite{moore1981principal} and has been explored for continuous and discrete time \cite{pernebo1982model}. A clear presentation is provided in \cite{dullerud2000course}. The basic idea is that a change of basis is used to order the states equally observable and controllable, and to order them from most to least controllable and observable, and then the least controllable and observable states are truncated; see \cite{benner2015survey} for a survey of projection-based reduction techniques, including balanced truncation, for parameterized dynamical systems. A complete  explanation of BT is presented in Section \ref{sec:bt}.

Perturbation theory is a well studied area and has a rich background in linear operator theory and the controls field \cite{kokotovic1976singular,kato1995perturbation,kokotovic1999singular
}, and it has seen renewed interest in recent years \cite{pare2015unified,jardon2016analysis,jardon2017stabilization}. It is commonly applied in the context of well-separated time scales.  In this case the ratio of time-scales identifies an explicit ``small'' parameter in which a series expansion can be computed.  The theory has also been applied to balanced realizations
, which we will refer to as Balanced Singular Perturbation Approximation (BSPA) \cite{fernando1982singular,fernando1982singular2,
fernando1983singular,liu1989singular}. 
In \cite{fernando1982singular2}, an alternative parameterization of the transfer function is used to provide a different unification of BT and BSPA using the ``generalized singular perturbational approximation.'' This parameterization is quite different than the one proposed here and does not lead to the insight that BT and BSPA occur on the boundary of a particular manifold representing the original model class. A complete  explanation of BSPA is presented in Section \ref{sec:spa}.  
The Manifold Boundary Approximation Method, 
introduced in \cite{Transtrum2014}, is an algorithm that, given a parameterized model class, systematically generates a ``representative" subset of the class parameterized by fewer parameters than the original set of models. The process is enabled by the fact that a parameterized model can be interpreted as a mapping between a parameter space and prediction space.  The set of all possible models generated by varying the parameters over all values typically corresponds to a manifold with the model mapping acting as a coordinate chart.  To make this abstract idea more concrete, consider a model class with $N$ parameters that is sampled at $M$ different points (times or frequencies). In this case, the model defines a mapping from $\RR^N$ to $\RR^M$ whose image is an $N$ dimensional submanifold of $\RR^M$.  We illustrate this conceptually for several simple models in this manuscript.  Each point on the manifold corresponds to the sample of a different set of parameters. For practical problems, explicitly constructing the entire manifold would be impractical, but such a manifold can be explored using computational differential geometry.  Systematic study of model manifolds from a variety of classes has revealed that they are typically bounded and that the boundary consists of a hierarchical cell complex, i.e., a hierarchy of faces, edges, corners, etc., similar to a high-dimensional polygon.  Each boundary cell is a manifold that corresponds to a model class with fewer parameters.  The ``best" approximation of this model can be found by locating the closest boundary point with the desired level of complexity, where closest depends on your metric of choice. MBAM will be formalized in Section~\ref{sec:mbam}.

In this paper we explore the problem of system approximation and identify an intrinsic structure to the problem by proposing an appropriate parameterization for LTI systems that enables us to show that BT and BSPA can be derived by applying MBAM to it. Given this insight, and since MBAM does not depend on linearity, only on a given parameterization, we conjecture that MBAM may provide a framework for the systematic model reduction of more general (nonlinear) systems.
In \cite{andersson1999model}, Andersson \textit{et al.} provide bounds for model reduction on systems that include uncertainties or nonlinearities, modeled as a delta block, assuming the delta block can be modeled using integral quadratic constraints (IQC). The authors also show that BT and BSPA appear as natural limits of the proposed IQC model, but do not provide any insight as to how to interpolate between the two. While the model in \cite{andersson1999model} is more general than the model considered herein, we believe that the ideas proposed in this work can be applied to a broader classes of models, independent of the IQC assumption. 

In Section \ref{sec:back} we review Balanced Truncation and Balanced Singular Perturbation Approximation. We present MBAM in Section \ref{sec:mbam}. In Section \ref{sec:param}, we present a parameterization of LTI systems, and then in Section \ref{sec:main} we use it to derive BT and BSPA using MBAM. In Section \ref{sec:example} we present a simple example that illustrates the result and how MBAM can give insight into model reduction of linear systems, allowing a way to interpolate between BT and BSPA. 
A condensed version of the work is given in \cite{pare2015unified}. 
Further contributions of this paper include: 1) the inclusion of the complete proofs; 
 2) the addition of several illustrative examples;
 3) a clearer, more in-depth presentation of MBAM; and
4) an extensive discussion of the parameterization of LTI systems. 

\section{Model Reduction of LTI Systems}\label{sec:back}

In this work we consider LTI systems described by,
\begin{equation}\label{eq:ss}
\begin{split} 
\dot{x}(t) & = Ax(t) + Bu(t) \\
y(t) & = Cx(t) + Du(t)
\end{split}
\end{equation}
where $x(t)\in \mathbb{R}^n$, $y(t)\in \mathbb{R}^p$, and $u(t)\in \mathbb{R}^m$. 
We will assume the system described by \eqref{eq:ss} is minimal and stable, and note that it generates  a unique transfer function given by:
\begin{equation}\label{eq:tf}
G(s) = C(sI-A)^{-1}B+D 
\end{equation}
which is a $p\times m$ matrix of rational functions in the complex variable $s\in\mathbb{C}$.  Note that while (\ref{eq:ss}) uniquely specifies (\ref{eq:tf}), there are many equivalent realizations  of the form (\ref{eq:ss}) for any  given input-output map (\ref{eq:tf}).


Model reduction is an important, well--studied problem. Although there are many approaches to model reduction, this work focuses on two  types of model reduction for LTI systems: BT and BSPA. We will review these two methods and their accompanying literature.

\subsection{Balanced Truncation}\label{sec:bt}

Consider a  stable system given by a bounded map (\ref{eq:tf}), and let  some (arbitrary) minimal realization of it, (\ref{eq:ss}), be known.   In general, the description (\ref{eq:ss}) may not reveal  any particular insight about the system, compared with any other realization of the system, and all  such minimal realizations  can be found by changing coordinates of the state variable, i.e. by rewriting  (\ref{eq:ss}) in terms of a new basis of the state space obtained by multiplying $x$ by any invertible matrix,  $\bar{x}=Tx$, and considering the system in terms of $\bar{x}$; we call $T$ a state transformation. 

Nevertheless, while an arbitrary  realization may be no more useful than any other, some  particular realizations are  more useful for particular types of analysis than others.   For example, modal realizations change basis of the state space to reveal the eigenstructure of each dynamic mode of the system.  Any such ``particular" realization is called {\em canonical} when there is a well-defined process for constructing it, i.e. when such a realization always exists, and when it is unique for every system $G$.  Thus, a canonical realization is in one-to-one correspondence with  any given  input-output map (\ref{eq:tf}).

It turns out that there is a canonical realization of LTI systems that has a property  we call  {\em balanced}.  That is to say, although there may be multiple realizations  with the  property of being balanced, we can  specify a fixed procedure that will  repeatably construct such a realization  for any system $G$.  We will represent this realization by:
\begin{equation}\label{eq:ssbal}
\begin{split}
\dot{\bar{x}}(t) & =  \bar{A}\bar{x}(t) + \bar{B}u(t) \\
y(t) & =  \bar{C}\bar{x}(t) + \bar{D}u(t)
\end{split}
\end{equation}
where $\bar{A}$ satisfies the Lyapunov equations: 
\begin{equation}\label{eq:lap}
\begin{split}
\bar{A}^T X + X \bar{A} &= - \bar{C}^T \bar{C} \ \ \text{  and} \\ 
\bar{A}X + X\bar{A}^T &= - \bar{B}\bar{B}^T
\end{split}
\end{equation}
with $X=diag(\theta_1,\dots,\theta_n)$ and the $\theta_i$'s are the Hankel singular values (HSVs) of the system \cite{dullerud2000course}.  These Lyapunov equations are solved by the controllability and observability Gramians; since they are solved by the same diagonal matrix, $X$, we call the realization 
{\em balanced}, that is, each state  in this basis is as observable as it is controllable, quantified by the corresponding $\theta_i$. 
The HSVs are ordered from largest to the smallest in magnitude, and consequently the states are ordered 
from the most controllable/observable to the least. 

Once balanced, one may partition the states into two sets, a highly controllable/observable set and a low controllable/observable set.  Partitioning the state matrices commensurate with this partition of the states then yields:
\begin{equation}
\label{eq:partition}
\begin{split}
\bar{A} &= \begin{bmatrix}
\bar{A}_{11} \ \ \  \bar{A}_{12} \\
\bar{A}_{21} \ \ \  \bar{A}_{22} 
\end{bmatrix}, \\
\bar{B} &= \begin{bmatrix}
\bar{B}_1 \\
\bar{B}_2
\end{bmatrix}, \\
\bar{C} &= \begin{bmatrix}
\bar{C}_1 \ \ \ \bar{C}_2
\end{bmatrix}, \text{ and} \\
\bar{x} &= \begin{bmatrix}
\bar{x}_1 \\ \bar{x}_2
\end{bmatrix}
\end{split}
\end{equation}
where $\bar{A}_{11}\in \mathbb{R}^{(n-k) \times (n-k)}$, $\bar{A}_{22}\in \mathbb{R}^{k \times k}$, and the rest of the blocks are the appropriate dimensions.
BT of $k$ states projects the system to become: 
\begin{equation*}
\begin{split}
\dot{\bar{x}}_1(t) & =  \bar{A}_{11} \bar{x}_1(t) + \bar{B}_1u(t) \\
y(t) & =  \bar{C}_1 \bar{x}_1(t) + {D}u(t).
\end{split}
\end{equation*}
Note that stability is preserved under truncation \cite{dullerud2000course}. 
The transfer function for the truncated system becomes
\begin{equation*}
\bar{G}_r(s) = \bar{C}_1(sI-\bar{A}_{11})^{-1}\bar{B}_1+{D}.
\end{equation*}
Using the H-infinity norm, if $\theta_{n-k+1}<\theta_{n-k}$ and $\theta_{n-k+1}=\cdots=\theta_n$,
\begin{equation}\label{eq:err1}
\| G - \bar{G}_r \|_{\infty} \leq 2\theta_{n-k+1}.
\end{equation} 
If $\theta_{n-k+1}<\theta_{n-k}$ and $\theta_{n-k+1}>\cdots > \theta_n$,
\begin{equation}\label{eq:err2}
\| G - \bar{G}_r \|_{\infty} \leq 2\sum_{i=n-k+1}^n \theta_i.
\end{equation} 
Proofs of these error bound can be found in \cite{glover1984all,enns1984model,enns1984modelphd,
dullerud2000course}. BT has been extended to several different classes of systems including time-varying, multidimensional, and uncertain systems \cite{beck1996model,lall1999model,beck2006coprime}, and  many structure preserving model reduction techniques, with varying definitions, have been considered as extensions of BT \cite{hsu1983decomposition,zhou1995structurally,
vandendorpe2004model,li2005structured,sandberg2009model}. Extensions have also been developed for nonlinear systems \cite{scherpen1993balancing,scherpen1996balancing,lall2002subspace}.

\subsection{Balanced Singular Perturbation Approximation} \label{sec:spa}

Given a partitioned balanced realization characterized as in \eqref{eq:partition}, the reduced system given by BSPA is 
\begin{equation*}
\begin{split}
\dot{\hat{x}}_1(t) & =  \hat{A}_{11} \bar{x}_1(t) + \hat{B}_1u(t) \\
y(t) & =  \hat{C}_1 \hat{x}_1(t) + \hat{D}u(t)
\end{split}
\end{equation*}
where
\begin{equation}
\label{eq:spa}
\begin{split}
\hat{A} &=  \bar{A}_{11} - \bar{A}_{12}\bar{A}_{22}^{-1} \bar{A}_{21}, \\
\hat{B}  &=  \bar{B}_1 - \bar{A}_{12}\bar{A}_{22}^{-1} \bar{B}_2, \\
\hat{C} &= \bar{C}_1 - \bar{C}_2 \bar{A}_{22}^{-1} \bar{A}_{21}, \text{ and}\\
\hat{D} &= \bar{D} - \bar{C}_2 \bar{A}_{22}^{-1} \bar{B}_2.
\end{split}
\end{equation}
The matrix $\hat{A}$ is the Schur complement of $\bar{A}$, which we will denote by $\bar{A}/\bar{A}_{22}$. 
Note, similarly to BT, stability is preserved under the BSPA process \cite{liu1989singular}. 

Define the transfer function of the BSPA system to be
\begin{equation*}
\hat{G}_r(s) = \hat{C}(sI-\hat{A})^{-1}\hat{B}+\hat{D}.
\end{equation*}
Then, identically to BT, if $\theta_{n-k+1}<\theta_{n-k}$ and \newline
$\theta_{n-k+1}=\cdots = \theta_n$,
\begin{equation}\label{eq:err3}
\| G - \hat{G}_r \|_{\infty} \leq 2\theta_{n-k+1}.
\end{equation} 
If $\theta_{n-k+1}<\theta_{n-k}$ and $\theta_{n-k+1}>\cdots > \theta_n$,
\begin{equation}\label{eq:err4}
\| G - \hat{G}_r \|_{\infty} \leq 2\sum_{i=n-k+1}^n \theta_i.
\end{equation} 
These bounds are proven in \cite{liu1989singular}.

In many respects, BT and BSPA are complementary types of approximations.  They are both derived from the same block partition of a balanced realization.  They share the same 
error bounds, illustrated in \eqref{eq:err1}-\eqref{eq:err2} and \eqref{eq:err3}-\eqref{eq:err4}. 
It is well known that BT typically provides better approximations at high frequencies while BSPA works well at low frequencies \cite{liu1989singular}.  In what follows, we will see how these similarities allow us to unify both approximations as limiting approximations of a balanced system.

\section{Manifold Boundary Approximation Method} \label{sec:mbam}

The Manifold Boundary Approximation Method, i.e. MBAM, 
was originally described in \cite{Transtrum2014}, but we present an overview here for completeness.  The idea is that a parameterized set of models can be viewed as a mapping between its parameter space and a prediction space, sometimes called ``data space", which is the space of sampled, measured outputs.  As such, it is natural to interpret a model class, or a continuously parameterized set of models, as a manifold embedded in the space of possible predictions (i.e. embedded in data space).  We refer to this manifold as the \emph{model manifold}, denoted by $\MM(\DD)$. 


Control-oriented models, whether linear or nonlinear, are often written as differential algebraic equations, or DAEs:
 \begin{eqnarray}
    \label{eq:dxdt}
 	\dot{\mathbf{x}} & = & \mathbf{f}(\mathbf{x}, \mathbf{z}, \mathbf{p}, \mathbf{u}, t) \\
    \label{eq:g}
    0 & = &\mathbf{g}(\mathbf{x}, \mathbf{z}, \mathbf{p}, \mathbf{u}, t),
    \end{eqnarray}
where $\mathbf{x}$ is the vector of (differential) state variables, $\mathbf{z}$ are the algebraic variables, $\mathbf{p}$ are parameters, $\mathbf{u}$ are inputs (typically assumed to be known in estimation studies) and t is the (scalar) time variable.  For example, \cite{seanbook} illustrates how DAEs with differentiation index zero can be used to represent interconnections of systems.  Likewise, in general we assume the system measurement vector is of the form: 
\begin{equation}
\label{eq:y}
\mathbf{y} = \mathbf{h}(\mathbf{x}, \mathbf{z}, \mathbf{p}, \mathbf{u}, t).
\end{equation}
Note that we begin our work with state space models because typically first-principles modeling techniques result in large sets of such systems of equations, often with more parameters than one can accurately estimate from available data, thus motivating the need for a reduction method.   
These equations lead to the following definition:

\begin{definition}
\label{def:modelmap}
A {\em model mapping} of the set of systems in \eqref{eq:dxdt}-\eqref{eq:y} sampled at fixed times is defined as the map 
$\MM~:~\DD~\subset~{\cal P} \rightarrow {\cal D}$, where ${\cal D}=\RR^{M\cdot p}$ is the data or prediction space, with $\bar{N}<M$. 
\end{definition}

In this work, we assume that $f$, $g$, and $h$ in \eqref{eq:dxdt}, \eqref{eq:g}, and \eqref{eq:y} are such that $\MM$ is smooth.  Furthermore, note that in Definition \ref{def:modelmap}, we have assumed that $\bar{N}<M$, with no bound on $M$.  It is possible for the data or prediction space to be infinite dimensional.  In many cases, however, we sample the output at specific times, making the data space finite dimensional.  Note that while other methods, such as [48], sample system data to build an empirical controllability Gramian, the sampling here embeds the model manifold in an appropriately large data-space.  Moreover, the topological properties of the manifold are largely invariant to details of the sampling \cite{mark2014}.  
Thus, the model manifold becomes the central object to represent the parameterized model class describing the system.  

We use the notation $\MM(\DD)$ to denote the image of $\DD$, or the set of points in $\RR^{M\cdot p}$ that are mapped from $\DD$; this set of points is a manifold embedded in $\cal D$, called the {\em model manifold}.  Each point on the manifold, then, corresponds to a particular choice of parameter, $p\in \cal P$, which specifies a given system in the model class.  For a given model manifold $\MM(\DD)$, we denote the closure of the manifold as $\MC$. 
\begin{definition}
\label{def:interior}
Given a model mapping $\MM:\mathrm{D} \subset {\cal P} \rightarrow {\cal D}$, a point $d \in \MC$ is {\em interior} if there exists an open neighborhood of $\MC$ centered at $d$.  If $d$ is not interior then it is a {\em boundary point}.  
\end{definition}

The set of boundary points of the manifold closure defines the {\em boundary} of the manifold, denoted by $\MB$.
\begin{definition}
\label{def:MBAM}
Given a model mapping $\MM:\DD \subset {\cal P} \rightarrow {\cal D}$, a model mapping $\MMR$ is a $k^{th}$-order {\em manifold boundary approximation} of $\MM$ if:
\begin{enumerate}
\item $\MMR:\DDR \subset \tilde{{\cal P}}=\RR^{\bar{N}-k} \rightarrow {\cal D}$ and 
\item $\MMR(\DDR) \subset \MB$.
\end{enumerate}
\end{definition}
That is to say, a  $k^{th}$ order model boundary approximation is itself a manifold defined as being on the boundary of the model manifold and having $k$ less parameters 
than the original model manifold.  A model boundary approximation {\em method}, then, is a systematic process for choosing a $k^{th}$-order approximation to the model mapping on the boundary  of the model manifold.

Note that in the case where the output is sampled, one needs to collect more samples than parameters to ensure that the model manifold is embedded in a large enough space.  In general, though, as long as there are enough sample points, the topological features of the resulting model manifold are invariant to the specific times when the samples are taken as long as they are decided randomly--see \cite{mark2014} for details.  Once a model manifold is constructed, however, we then turn our attention to the specific details of choosing a $k^{th}$-order approximation.  

To develop such an MBAM method, we consider the sensitivity of such models to their parameters by taking partial derivatives as follows:
\begin{eqnarray}
	\label{eq:dxdp}
	\frac{d}{dt} \frac{\partial \mathbf{x}}{\partial \mathbf{p}} & = & \frac{\partial \mathbf{f}}{\partial \mathbf{x}} \cdot \frac{\partial \mathbf{x}}{\partial \mathbf{p}} + \frac{\partial \mathbf{f}}{\partial \mathbf{z}} \cdot \frac{\partial \mathbf{z}}{\partial \mathbf{p}} + \frac{\partial \mathbf{f}}{\partial \mathbf{p}} \\
    \label{eq:dgdp}
    0 & = & \frac{\partial \mathbf{g}}{\partial \mathbf{x}} \cdot \frac{\partial \mathbf{x}}{\partial \mathbf{p}} + \frac{\partial \mathbf{g}}{\partial \mathbf{z}} \cdot \frac{\partial \mathbf{z}}{\partial \mathbf{p}} + \frac{\partial \mathbf{g}}{\partial \mathbf{p}} \\
    \label{eq:dydp}
    \frac{\partial \mathbf{y}}{\partial \mathbf{p}} & = & \frac{\partial \mathbf{h}}{\partial \mathbf{x}} \cdot \frac{\partial \mathbf{x}}{\partial \mathbf{p}} + \frac{\partial \mathbf{h}}{\partial \mathbf{z}} \cdot \frac{\partial \mathbf{z}}{\partial \mathbf{p}} + \frac{\partial \mathbf{h}}{\partial \mathbf{p}},
\end{eqnarray}
Note that these equations are linear in terms of sensitivities, but the
matrices involved do vary with time. We see
that \eqref{eq:dxdp}-\eqref{eq:dydp} are derived by differentiating \eqref{eq:dxdt}-\eqref{eq:y} with respect
to the parameters and applying the chain rule to account for the
implicit dependence of the dynamic, algebraic, and measured
variables on the parameters. We also use the second
order sensitivities in the sequel, the equations for which can be derived similarly. The problem of calculating parametric
sensitivities for dynamical systems is well-known, with a long
history \cite{leis1988simultaneous}. However, deriving expressions for the first and second order sensitivities by hand can be tedious and error prone
(particularly for large models). We therefore use automatic symbolic and numerical differentiation \cite{rall1981automatic,bucker2006automatic} to simplify the process.

Previous work suggests that model manifolds for real, physical systems tend to be ``thin,'' in the data space where they are embedded; that is to say, measured variables tend to be considerably less sensitive to some parameter combinations than others \cite{Transtrum2010, Transtrum2011, Transtrum2014}. MBAM exploits this ``sloppy" character of such model classes by approximating this ``thin" manifold by its boundary.  We identify the boundary by numerically computing a geodesic on the model manifold and using
the results from this calculation to identify an approximate
model.  This ``boundary" model will have fewer parameters than the original, and it approximates the original model manifold the same way a long edge of a ribbon approximates the ribbon. 

Geodesics are
the analogs of straight lines generalized to curved surfaces,
and we compute them numerically as the solution to a second order
ordinary differential equation in parameter space (while utilizing
quantities from the data space):
\begin{equation}
 \label{eq:geodesic}
 \frac{d^2 p^i}{d \tau^2} = \sum_{j,k} \Gamma^i_{jk} \frac{d p^j}{d \tau} \frac{d p^k}{d \tau}; \ \ 
 \Gamma^i_{jk} = \sum_{l,m} = \left( \mathbf{I}^{-1} \right)^{il} \frac{\partial y_m}{ \partial p_l} \frac{\partial^2 y_m}{\partial p_j \partial p_k}
\end{equation}
where $\Gamma$
are the Christoffel symbols \cite{Spivak1979,Misner1973, Transtrum2011}, containing
curvature information about the mapping between parameter
space and data space.  These are expressed in terms of the first
and second order parametric sensitivities~\eqref{eq:dxdp}-\eqref{eq:dydp} and
$\mathbf{I}$ is the Fisher Information Matrix (FIM) for the measurement process. The geodesic is
parameterized by 
$\tau$, which is proportional to the arc length of the geodesic
as measured on the model manifold, i.e., in data space. Solving~\eqref{eq:geodesic} gives a parameterized curve
$\mathbf{p}(\tau)$
in parameter space that is
used to reveal a limiting behavior in the model as we demonstrate
below.

Equation~\eqref{eq:geodesic} is an ordinary differential equation that we solve
as an initial value problem. Here, we take the model’s
nominal, or ``true" parameter values as the starting point of the geodesic.
The initial ``velocity'' is taken to be the least sensitive parameter
direction as measured by the FIM. In this way, the MBAM procedure can be summarized as a five step algorithm. Here, we describe the algorithm and give a summary in Fig.~\ref{fig:algorithm}. 

\begin{figure}
     \centering
      \includegraphics[width=0.75\columnwidth]{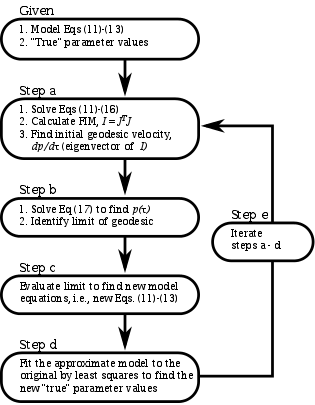}
      \caption{The Manifold Boundary Approximation Method can be summarized as a five step, iterative algorithm.}
\label{fig:algorithm}
\end{figure}

\textbf{Step a}:
The least sensitive parameter combination is
identified from an eigenvalue decomposition of the FIM which becomes  the  initial  geodesic  velocity,
$\partial \mathbf{p} / \partial \tau$
The geodesic acceleration
$\partial^2 p/ \partial \tau^2$
is given by \eqref{eq:geodesic}. If
$\sum_i (\partial p^i/\partial \tau)(\partial^2 p^i/\partial \tau^2) < 0$, then reverse the initial velocity:
$\partial \mathbf{p}/\partial \tau \rightarrow - \partial \mathbf{p}/ \partial \tau$. This heuristic resolves the ambiguous
direction associated with the eigenvalue.

\textbf{Step b}: A geodesic on the model manifold is constructed by
numerically solving~\eqref{eq:geodesic} using the nominal or "true'' parameter values and
the velocity $\partial \mathbf{p} / \partial \tau$
calculated in \textbf{Step a} as initial conditions.
The limiting behavior of this curve identifies the boundary of
the model manifold.

\textbf{Step c}: Having found the edge of the model manifold, the
corresponding lower-dimensional model class is identified as an approximation to the
original model class. By inspecting which values of the parameter
vector become infinite, we identify the boundary as a limiting
approximation in the model class. We evaluate this limit to construct
the approximate model class parameterized by fewer parameters.

\textbf{Step d}: We choose parameter values of this approximate model class by fitting the approximate model to the behavior of
the original model using, for example, least squares regression.

\textbf{Step e}: The procedure is repeated until the reduced model is
unable to approximate quantities of interest within the desired tolerance.

Note that there can be many model boundary approximation methods, depending on which boundary one chooses as an approximation of the model manifold.  A common approach is to choose a metric and then select the boundary that is closest to a given point  in the data space.  In the original presentation of MBAM \cite{Transtrum2014}, an information distance was imposed on the data space defined by the FIM, which defines a Riemannian metric on the parameter space, \cite{Spivak1979,Misner1973}. From this Riemannian metric, computational differential geometry can be used to identify candidate boundary approximations \cite{Transtrum2014}.  Nonetheless, MBAM is a topological operation in general and is agnostic to the actual choice of metric \cite{mark2014}.

\subsection{Michaelis Menten Example} \label{sec:simpExample}

\begin{figure}
    \centering
     \begin{subfigure}[b]{0.5\textwidth}
     \centering
      \begin{overpic}[width=\columnwidth]{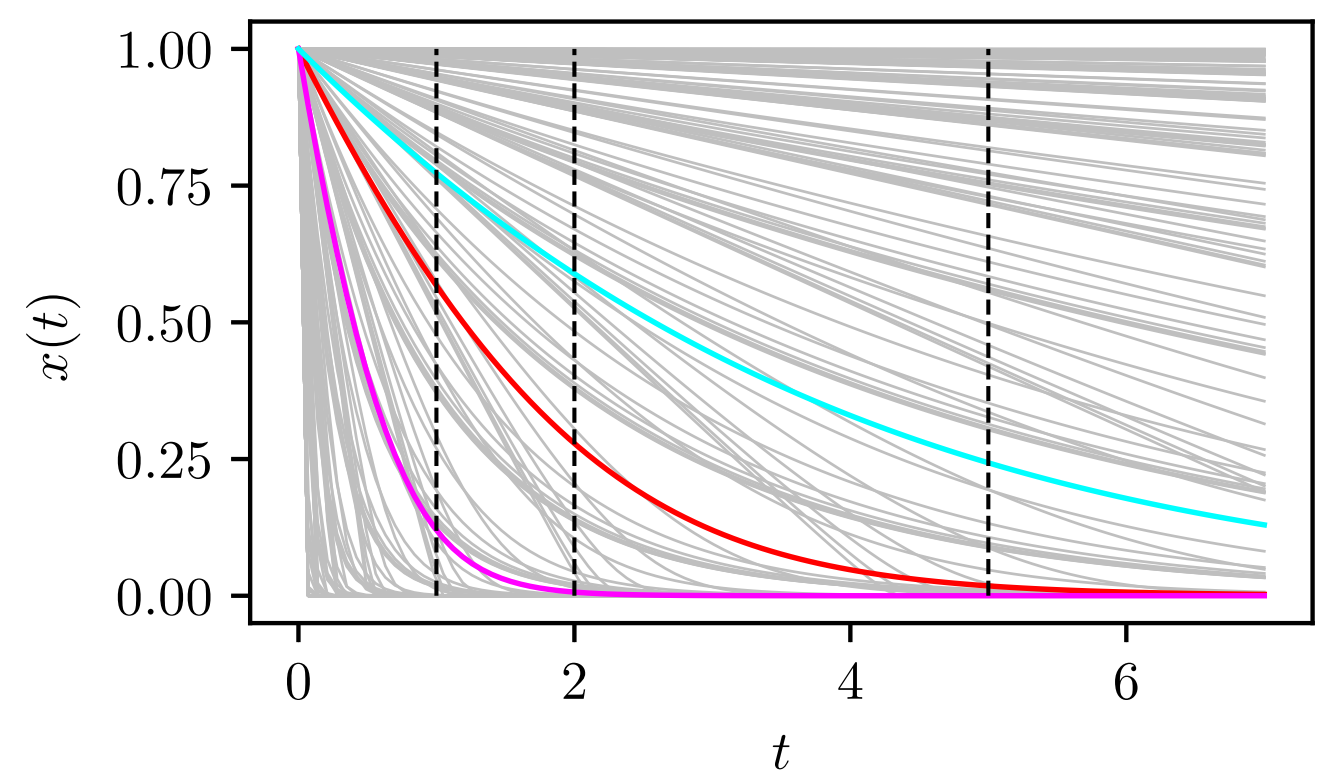}
\end{overpic}
      \caption{Potential time series for the MMR model.  The red, cyan, and magenta curves are calculated by  $(\rho_1,\rho_2)$ values of $(1,1)$, $(3, 1)$, and $(1, 3)$ respectively.}
      \label{fig:AERtimeseries}
    \end{subfigure}
    \begin{subfigure}[b]{0.5\textwidth}
      \centering
      \begin{overpic}[width=\columnwidth]{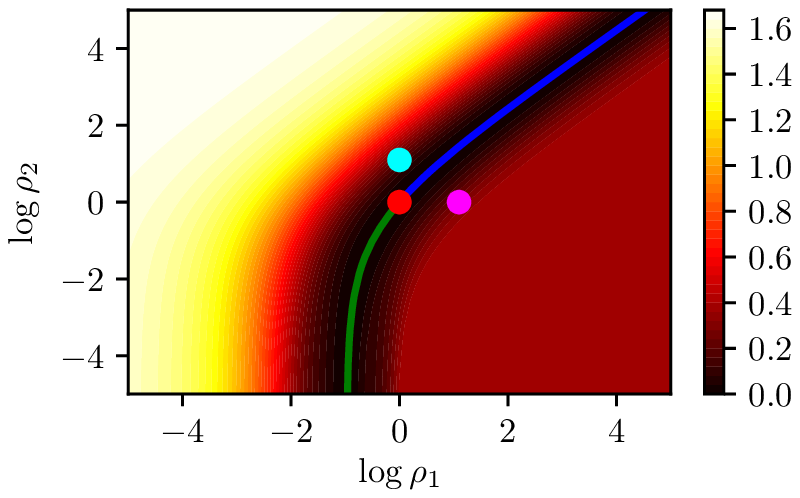}
\end{overpic}
      \caption{Visualizing the model's parameter dependence.  Color corresponds to the sum of squares of the the change in the model's prediction at each time point.  The red, cyan, and magenta points correspond to the curves of the same color in Fig.~\ref{fig:AERtimeseries}.  The green and blue curves are geodesic paths in opposite directions on the model manifold.}
      \label{fig:AERmodelmanifold}
    \end{subfigure}
    \begin{subfigure}[b]{0.5\textwidth}
      \centering
      \begin{overpic}[width=\columnwidth]{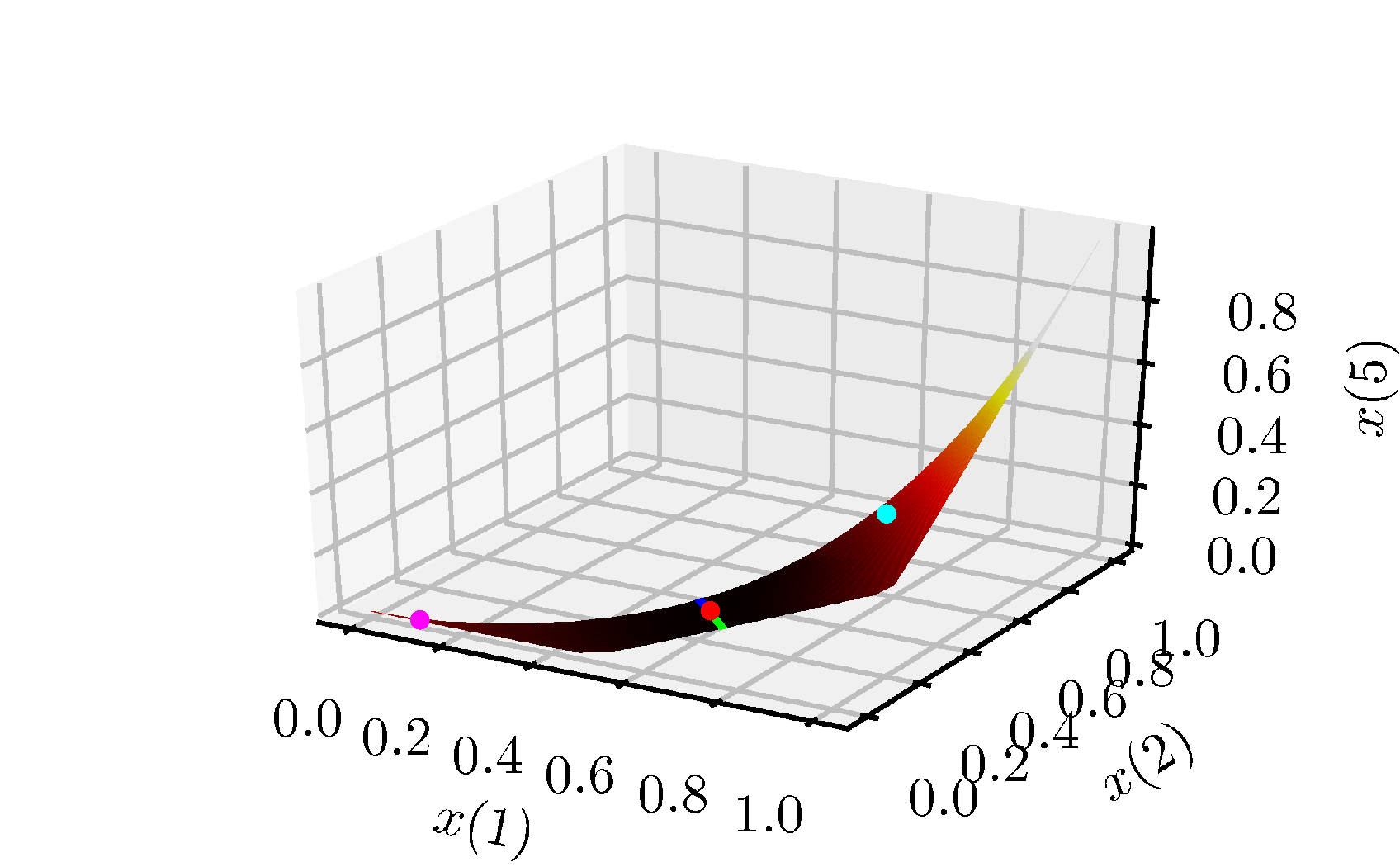}
\end{overpic}
      \caption{A view of the Model Manifold for the MMR model.  Colored points correspond to the point of the same color in Fig.~\ref{fig:AERmodelmanifold} and a trajectory of the same color in Fig.~\ref{fig:AERtimeseries}.  The geodesic curves in Fig.~\ref{fig:AERmodelmanifold} appear as short curves of the same color across the most narrow part of the manifold.}
      \label{fig:AERmodelmanifold2}
    \end{subfigure}
    \caption{MBAM works by sampling outputs at particular times for a variety of parameter values (Fig.~\ref{fig:AERtimeseries}), constructing geodesics in parameter space (Fig.~\ref{fig:AERmodelmanifold}), and interpreting the limiting case of these curves as portions of the manifold boundary in ``data space" (Fig.~\ref{fig:AERmodelmanifold2}).
    }
\label{fig:AER}
\end{figure}

The MBAM concept is best illustrated by example.  Consider the dynamics of 
the Michaelis Menten Reaction (MMR),
\begin{equation}
\label{eq:mmr}
\dot{x}(t) = \frac{-\rho_1 x(t)}{\rho_2 + x(t)}
\end{equation}
 where $\rho_i$ indicates a parameter in the dynamic equations characterizing the family of models under consideration \cite{mmr1,Briggs1925,mmr
}.  The parameters $\rho_i$ are only physically relevant for positive values; we therefore restrict our attention to the domain $\DD$ as the positive quadrant of the parameter space, that is, where $\rho_i\geq 0 \ \forall i$.  
Assuming a fixed initial condition, $x_0=1$, and varying $\rho_1$ and $\rho_2$ among all possible elements of the domain, $\mathrm{D}$, we observe some sample time series for this model as shown in Fig.~\ref{fig:AERtimeseries}.

Consider now three observations of this system at time points, $(t_1, t_2, t_3)$, indicated by the  vertical dotted lines in Fig.~\ref{fig:AERtimeseries}.  To visualize the sensitivity of the model to variations in the parameters, it is common to construct an ``error function" in parameter space as in Fig.~\ref{fig:AERmodelmanifold}.  An alternative approach is to recognize that the space of all possible experimental data at these three time points forms a three-dimensional ``data space.''    All possible model predictions, for different values of $\rho_1$ and $\rho_2$, correspond to a two-dimensional subset of this space.  This two-dimensional surface is the model manifold characterized as in Definition \ref{def:modelmap}. 
The manifold for the MMR model illustrated in Fig.~\ref{fig:AERmodelmanifold2} is bounded, a feature shown to be common among model manifolds with many more parameters than can be estimated from available data \cite{Transtrum2010, Transtrum2011}.  The existence of these boundaries is the crucial element that enables the manifold boundary approximation, where each boundary corresponds to a different model reduction of the original model in \eqref{eq:mmr}. The ``top''  boundary reached by the blue geodesic curve,  
corresponds to the limit where $\rho_1,\rho_2 \rightarrow \infty$, and the ratio $\rho_1/\rho_2$ becomes the new parameter. 
The ``bottom'' boundary reached by the green geodesic curve corresponds to the limit where $\rho_2 \rightarrow 0$, leaving $\rho_1$ as the sole parameter. These two different limiting approximations each lead to two distinct reduced models of the form:
\begin{align}
\dot{x}(t) &= - \frac{ \rho_1}{\rho_2} x(t), \text{ and} \label{eq:mmr2}\\
\dot{x}(t) &= - \rho_1, \label{eq:mmr3}
\end{align}
respectively.
A well--known interpretation of the MMR in biochemistry is that it describes a saturable reaction.  That is, at low concentrations (small values of $x$), the reaction rate is nearly first order (i.e., linear in $x$), while at high concentrations the rate saturates into a zeroth order (i.e., constant with respect to $x$).  Notice how these two limits are naturally identified as the boundaries of the model manifold in Fig.~\ref{fig:AERmodelmanifold2} and together bound the space of predictions that can be reached by the model.

To be clear, the use of MBAM to approximate the MMR system here reduced  the  set of systems from  a two dimensional set to a one dimensional set, but it did not reduce the dynamical order of the system, which started as a first order system and remained, in each limiting case, a first order system.  This is because parameter reduction does not necessarily lead to model-order reduction.  The next example will demonstrate a case  where parameter reduction does reduce the order of the model.  Although the dynamical order was not reduced in the example above, the reduced models are \emph{simpler}.   While the original model exhibited nonlinear dynamics, the approximate models in \eqref{eq:mmr2} and \eqref{eq:mmr3} are  linear and constant, respectively.  Note that code for computing each step of the MBAM process can be found at \cite{github}.

\subsection{Second-Order LTI Example} \label{sec:2ndExample}

\begin{figure}
    \centering
     \begin{subfigure}[h]{0.495\textwidth}
     \centering
     \begin{overpic}[width=\columnwidth]{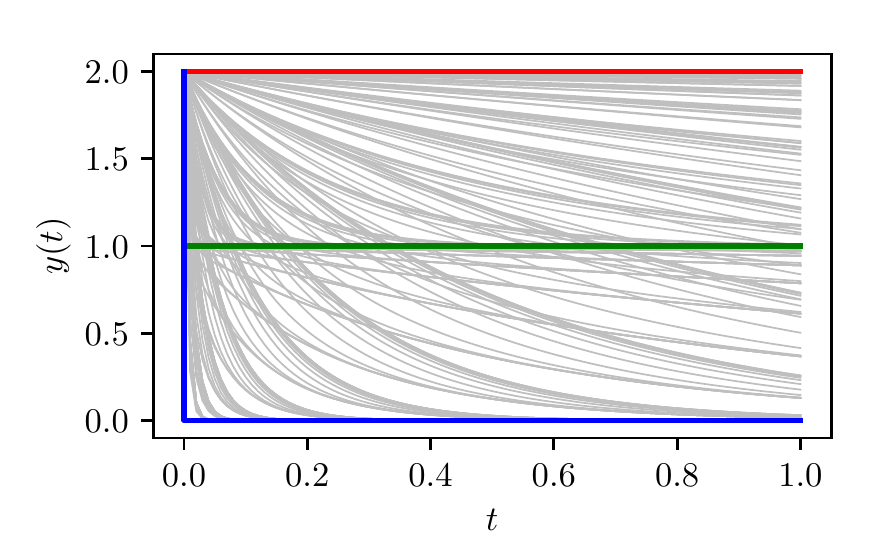}
\end{overpic}
      \caption{Different values of system parameters generate different output trajectories.  Sampling each trajectory at three (arbitrary) times associates a point in ${\mathbb R}^3$ (i.e. ``data space") with each parameter combination.}
      \label{fig:Slide2}
    \end{subfigure}
    \begin{subfigure}[h]{0.495\textwidth}
     \centering
       \begin{overpic}[width=\columnwidth]{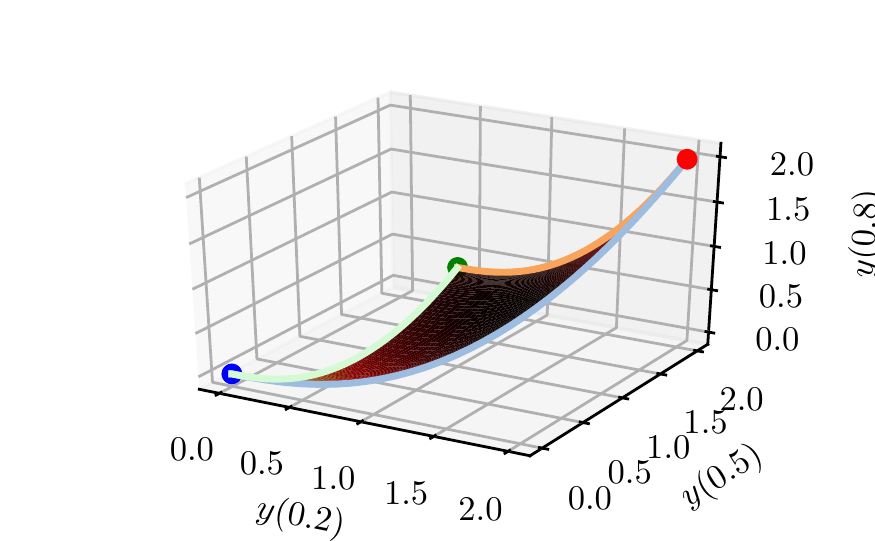}
\end{overpic}
      \caption{Plotting all possible parameter combinations generates the behavior of the set of systems.  Here are two perspectives of this behavior, or the ``model manifold," embedded in data space.  Different points are calculated by varying the parameters $\rho_1$ and $\rho_2$.}
      \label{fig:MM1}
    \end{subfigure}
    \begin{subfigure}[h]{0.495\textwidth}
     \centering
      \begin{overpic}[width=\columnwidth]{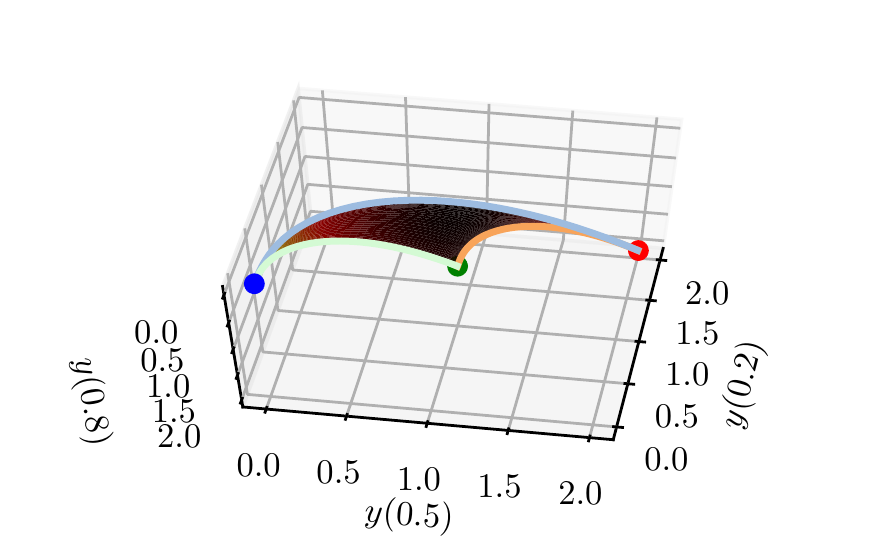}
\end{overpic}
      \caption{Plotting all possible parameter combinations generates the behavior of the set of systems.  Here are two perspectives of this behavior, or the ``model manifold," embedded in data space.  Different points are calculated by varying the parameters $\rho_1$ and $\rho_2$.}
      \label{fig:MM2}
    \end{subfigure}
    \caption{Considering the geometric properties of the model manifold yields important insights for generating low complexity approximations.
    }
\label{fig:MM}
\end{figure}
As a second example, consider the set of systems described by:
\begin{equation*}
\begin{split}
\begin{bmatrix}
\dot{x}_1(t) \\
\dot{x}_2(t)
\end{bmatrix} &= \begin{bmatrix}-\rho_1&0\\0&-\rho_2\end{bmatrix} \begin{bmatrix}x_1(t)\\x_2(t)\end{bmatrix}\\
y &= \begin{bmatrix}1&1\end{bmatrix} \begin{bmatrix}x_1(t)\\x_2(t)\end{bmatrix}
\end{split}
\end{equation*}
where $t, \rho_1, \rho_2\in \mathbb R^+$.  For any choice of $\rho_1$ and $\rho_2$, the behavior of the corresponding system is the set of functions
\[
y(t) = e^{-\rho_1 t}x_1(0)+e^{-\rho_2 t}x_2(0).
\]
When we fix the initial condition, say to $x(0)=\left[\begin{array}{cc}1&1\end{array}\right]^T$, this set is completely parameterized by the parameters $\rho_1$ and $\rho_2$.
One way to characterize the behavior of this set of systems is to associate each choice of parameters with an observation vector given by $v=\left[\begin{array}{ccc}y(t_1)&y(t_2)&y(t_3)\end{array}\right]^T$, where the observation times $t_1$, $t_2$, and $t_3$ are fixed.  Spanning over all choices of parameters $\rho_1$ and $\rho_2$, the observation vectors $v(\rho_1,\rho_2)$ then sweep out a two-dimensional manifold embedded in $\mathbb{R}^3$, shown in Fig. \ref{fig:MM1}-\ref{fig:MM2}. 
This is the {\em model manifold} for this set of systems, and we note that although its size and shape somewhat depend on the observation times $t_1$, $t_2$, and $t_3$, its characteristic topological features (number of cusps and edges and the relationships among them) are determined solely by the relationship between the parameters, $\rho_1$ and $\rho_2$, and observations, $y(t)$.  That is to say, regardless of how complicated the underlying dynamics are, an $\bar{N}$-parameter system generates an $\bar{N}$-dimensional manifold embedded in a larger data space (a fact that  therefore requires the number of observation points to be larger than the number of parameters), and the relationship between these cusps and edges is fixed, regardless of when we make the observations or whether we make more than $n$ observations, etc. \cite{mark2014}.

This invariance of the key topological features of the model manifold allows us to consider a graphical representation among these features, shown in Fig. \ref{fig:2MM2}. 
\begin{figure}
     \centering
      \includegraphics[scale=.4]{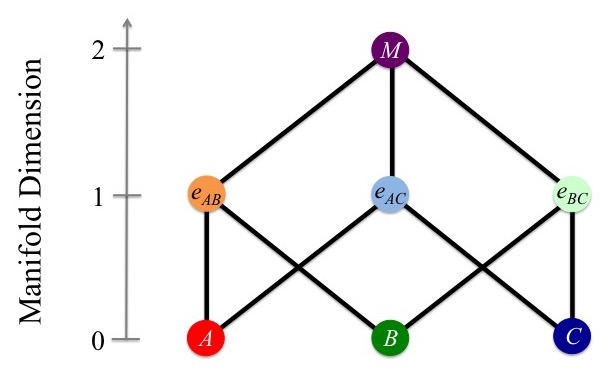}
      \caption{The model manifold $M$ in Fig.~\ref{fig:MM} has topological properties that are invariant to sufficiently rich observation processes.  This topology is captured by this graph, as the two-dimensional manifold has one-dimensional boundaries (edges) $e_{AB}$, $e_{AC}$, and $e_{BC}$, where each one-dimensional boundary is related to two zero-dimensional cusps indicated by the points $A$, $B$, and $C$.}
\label{fig:2MM2}
\end{figure}
This graph encodes the relationship between the model manifold surface in Fig.~\ref{fig:MM} and its boundaries, and between its boundaries and their endpoints.  Each of these objects are themselves a manifold of lower dimension, and this hierarchy of lower dimensional manifolds forms a partial ordering that reveals different ways the fully parameterized model can be approximated by representations of lower complexity.  

Thus, for example, each point on the model manifold, $M$, in Fig.~\ref{fig:MM} is specified by choosing particular values for both of the parameters $\rho_1$ and $\rho_2$; the fact that the model manifold is two-dimensional results from the fact that two parameters must be specified.  When one of these parameters is set to one of its limiting values, zero or infinity, the resulting set of models parameterized by the other parameter form a boundary on the model manifold, a boundary of dimension one (since only one parameter is left free).  Thus one edge is formed when one of the parameters (without loss of generality, assume it is $\rho_1$) is fixed at zero, yielding the remaining behaviors $y(t) = 1+e^{-\rho_2 t}$ represented by the one-dimensional manifold, the edge $e_{AB}$.  Another edge is formed when $\rho_2$ approaches infinity, yielding $y(t) = e^{-\rho_1 t}$ and represented by the edge $e_{BC}$.  Finally, the last boundary of $M$ is formed when $\rho_1$ is fixed to the same value as $\rho_2$, yielding $y(t) = 2e^{-\rho_2 t}$, and is represented by the edge $e_{AC}$.  

The three points where these edges intersect form zero-dimensional manifolds and represent systems where both parameters are fixed at one of the limiting values.  At point $A$, $\rho_1=\rho_2=0$ and the output $y(t)\equiv 2$. At point $B$, one of the parameters ($\rho_1$ in our case above) is zero while the other approaches infinity, with $y(t)\equiv 1$.  At point $C$ both parameters approach infinity and $y(t)\equiv 0$.  

The partial ordering among these models of decreasing complexity is captured by the Hasse diagram, \cite{mark2014}, shown in Fig.~\ref{fig:2MM2}; this diagram captures the relationship between model complexity and different types of approximations.  So, for example, $e_{AB}$, with behavior $y(t) = 1+e^{-\rho_2 t}$, is an approximation that considers that the unmodeled mode of the system is much slower than the other; it models the fast dynamics of the system.  On the other hand, $e_{BC}$, with behavior $y(t) = e^{-\rho_1 t}$, is an approximation that considers that the unmodeled mode of the system is much faster than the other and has already reached equilibrium; it models the slow dynamics of the system.  Finally, $e_{AC}$ describes the situation where both decay rates are comparable; it models the dynamics of a symmetric system.  

Each of these three types of approximations capture a different idealized behavior of every model in the set of systems parameterized by $\rho_1$ and $\rho_2$.
Likewise, each of these approximations are themselves approximated at their endpoints by one of three extreme behaviors $A$, $B$, or $C$.  These behaviors are themselves stationary, describing the regimes of behavior characterizing these systems.  In this example, $A$ represents a ``high" value, $B$ a ``medium" value, and $C$ a ``low" value, but in other contexts they might represent different phenotypes or behavior regimes such as ``healthy" and ``sick," etc.

These examples all illustrate the underlying mechanism driving MBAM techniques--that of looking for approximations on the boundary of the model manifold.  Current research is exploring efficient algorithms for conducting such a search, as well as application of the ideas to large scale systems, \cite{mark2014}, 
but this work reports how both BT and BSPA, and a host of hybrid techniques, become special cases of MBAM methods when applied to LTI systems.
\section{Parameterizations of LTI Systems}\label{sec:param}

The previous examples show 
how MBAM can reduce the number of parameters used to describe a set of models. The reduced parameter descriptions for the first example, given by \eqref{eq:mmr2} and \eqref{eq:mmr3}, 
parameterize a subset of the original model class, but this subset is chosen to be ``representative" of the original class in the sense that it is a set of models on the boundary of the original set.  We can parameterize this set of boundary models with fewer parameters than the original set, hence MBAM is a parameter-reduction technique.

Parameter reduction is not necessarily model order reduction, however. The example in Section \ref{sec:simpExample} 
shows that MBAM can reduce the number of parameters without changing the order of the system, while the example in Section \ref{sec:2ndExample} shows that depending on the parameterization the order can be reduced. The interest of this work is to use MBAM to perform model order reduction on general LTI systems. This requires a careful examination of the parameters of such systems, so that reduced-order models generated by BT or BSPA can be shown to be instances of reduced-order models generated by MBAM. In this section we present and classify the parameters of general LTI systems, and the following sections will leverage this parameterization to prove this relationship between these two model order reduction techniques and provide a framework to interpolate between them.

\subsection{Parameterizations of Transfer Functions}

Consider the LTI system in \eqref{eq:ss}, which has $m$ inputs and $p$ outputs. The input-output behavior of such a system is characterized by
\begin{equation*}
	Y(s) = G(s)U(s)
\end{equation*}
where $U(s)$ and $Y(s)$ are the Laplace transforms of the input and output signals, respectively, and $G(s)$ is given by
\begin{equation}
\label{eq:ss2tf}
	G(s) = C(sI - A)^{-1}B + D.
\end{equation}
In this context, the operator $G$ is the \emph{transfer function} of the LTI system, and shows what the output of the system will be given any input. Thus $G$ is a $p \times m$ matrix, where each entry is a proper polynomial of the Laplace variable $s \in \mathbb{C}$, which we will write as
\begin{equation}
\label{eq:tf_mat}
	G(s) =
    \begin{bmatrix}
    	\frac{\nu_{11}(s)}{\delta(s)} & \cdots & \frac{\nu_{1m}(s)}{\delta(s)} \\
        \vdots & \ddots & \vdots \\
        \frac{\nu_{p1}(s)}{\delta(s)} & \cdots & \frac{\nu_{pm}(s)}{\delta(s)}
    \end{bmatrix}.
\end{equation}
Without loss of generality, we may consider that each entry has a common denominator $\delta(s)$, with degree $N$, and the degree of every polynomial $\nu_{ij}(s)$ is equal to $N$. 
The matrix $G$ can be rewritten as
\begin{equation}
\label{eq:tfparam}
	G(s) = \frac{\nu(s)}{\delta(s)}
\end{equation}
where $\nu(s) = [\nu_{ij}(s)]$. We assume that $\delta(s)$ is monic (the leading coefficient is one), since if it were not, in each entry, both $\delta(s)$ and $\nu_{ij}(s)$ could be divided by the leading coefficient in $\delta(s)$ without changing the behavior of $G(s)$. Thus, each $\nu_{ij}(s)$ is determined by the $N + 1$ coefficients; therefore $\nu(s)$ is completely characterized by $pm(N+1)$ coefficients. Adding in the $N$ coefficients of $\delta(s)$ (not $N+1$ because $\delta(s)$ is monic), it follows that $G(s)$ is completely characterized by $Npm + N + pm$ coefficients (see Fig. \ref{fig:standard_tf}).

\begin{figure}
	\centering
    \begin{subfigure}[h]{0.45\textwidth}
    	\centering
        \includegraphics[scale=0.4]{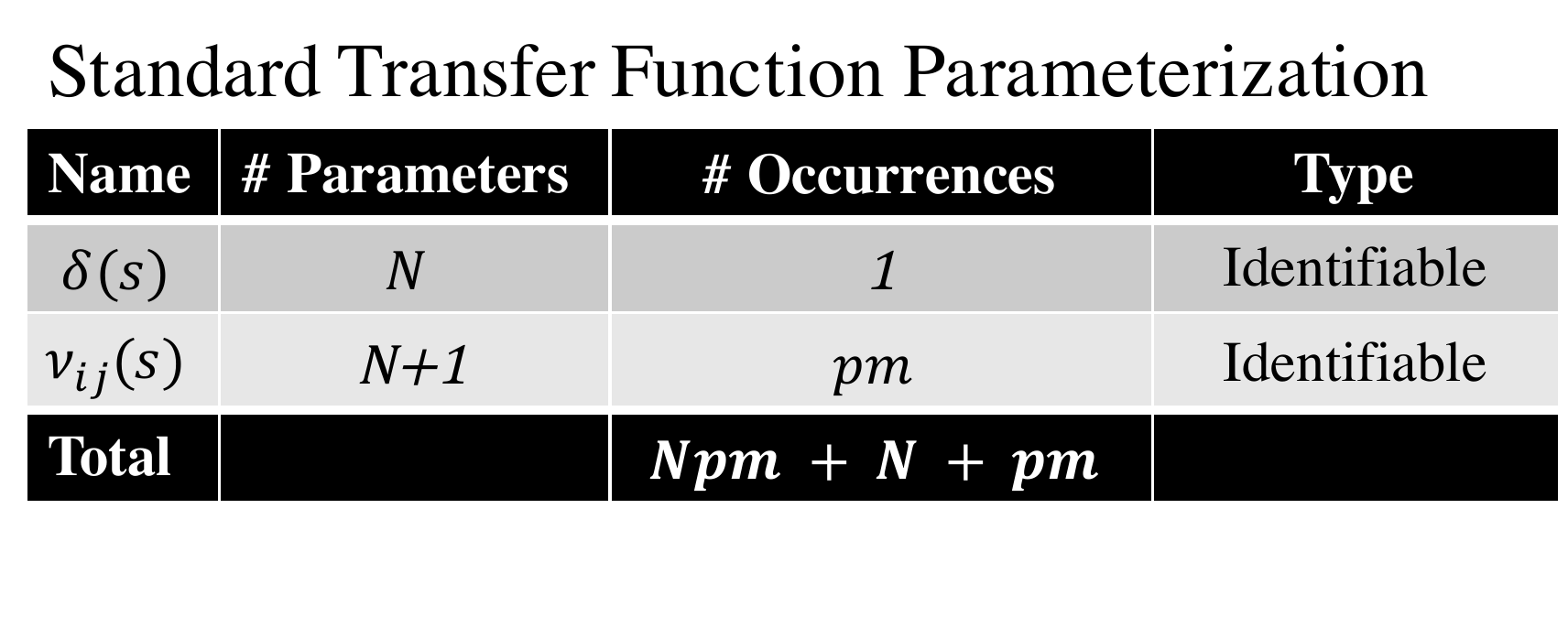}
        \caption{A minimal list of parameters that define the set of all transfer functions of degree $N$ with $m$ inputs and $p$ outputs. This parameterization is shown in \eqref{eq:tf_mat}.}
        \label{fig:standard_tf}
    \end{subfigure}
    \begin{subfigure}[h]{0.45\textwidth}
    	\centering
        \includegraphics[scale=0.4]{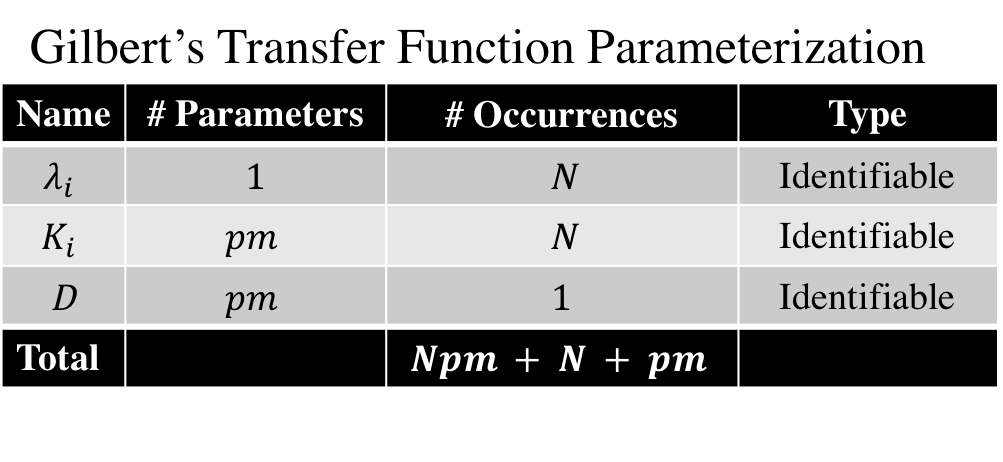}
        \caption{A minimal list of parameters that define the set of all transfer functions of degree $N$ with $m$ inputs and $p$ outputs, with the constraint that the poles of the system have geometric multiplicity one. This parameterization is shown in \eqref{eq:gilbert}.}
        \label{fig:gilbert_tf}
    \end{subfigure}
    \begin{subfigure}[h]{0.45\textwidth}
    	\centering
        \includegraphics[scale=0.4]{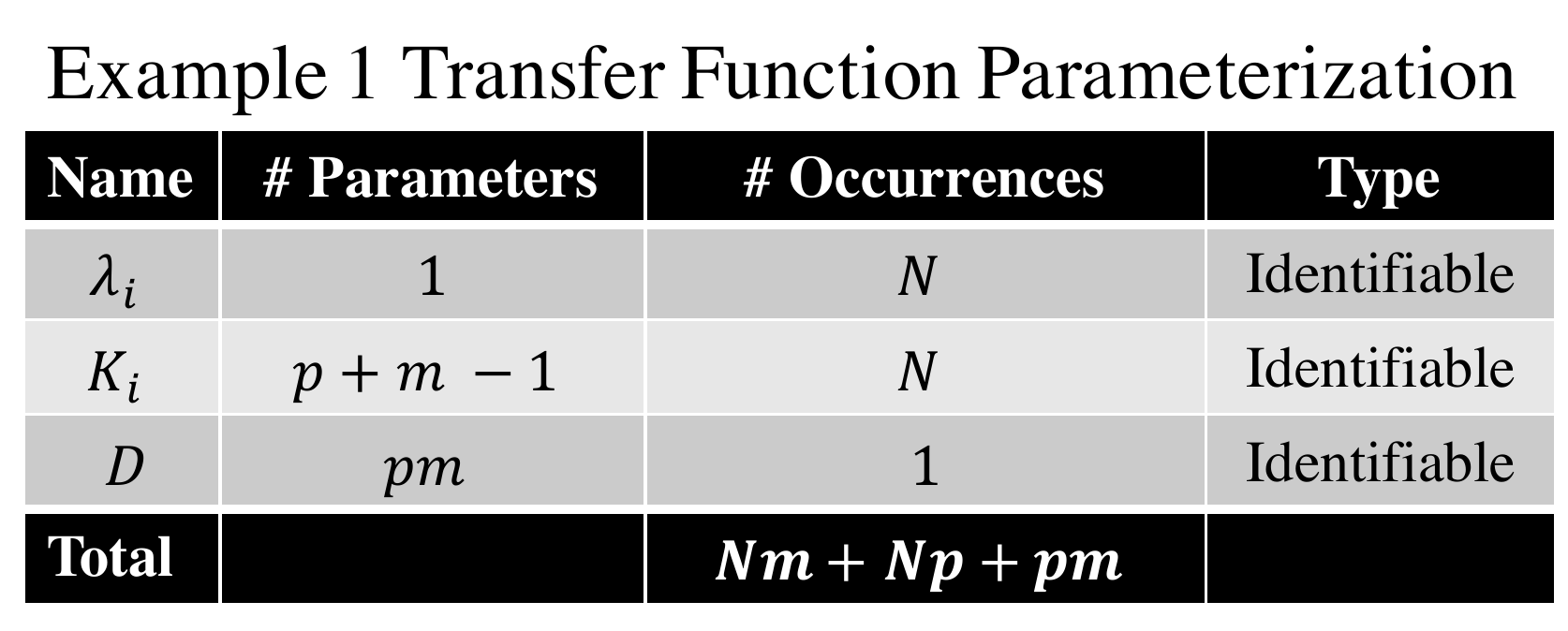}
        \caption{A minimal list of parameters that define the set of all transfer functions of degree $N$ with $m$ inputs and $p$ outputs, with the constraint that the system poles all have algebraic multiplicity one. This scenario is described in Example \ref{ex:r1}.}
        \label{fig:ex1_tf}
    \end{subfigure}
    \caption{Parameter counts for three different sets of transfer functions.}
\end{figure}

These coefficients provide a way to minimally parameterize the set of all $p \times m$ LTI operators of degree $N$. The parameterization is minimal in that if any coefficient is perturbed, then this perturbation will uniquely change the behavior of the LTI system, and no coordinated changes among the other parameters can produce the same change in the input-output behavior. Moreover, all LTI systems of size $p \times m$ with degree $N$ have such a parameterization. In this context we conclude that a $p \times m$ transfer function $G(s)$ with degree $N$ is uniquely characterized by $Npm + N + pm$ parameters.

Another way to minimally parameterize the set of $p \times m$ degree $N$ transfer functions is to employ a partial fraction expansion on each element $\frac{\nu_{ij}}{\delta(s)}$. For systems where every pole has a geometric multiplicity of one, the transfer function then becomes
\begin{equation}
\label{eq:gilbert}
	G(s) = D + \sum_{i=1}^N \frac{K_i}{s - \lambda_i}
\end{equation}
where $\lambda_i \in \mathbb{C}$ are the roots of $\delta(s)$, and $D, K_i \in \mathbb{R}^{p \times m}$. The matrix $D$ is a matrix of the leading coefficients for each $\nu_{ij}$ and is equal to $D$ in \eqref{eq:ss} for any state space realization of $G(s)$; indeed, every state realization of $G(s)$ will have the same $D$ matrix. The $K_i$ matrices contain the scalar numerators corresponding to $s - \lambda_i$ in the partial fraction expansion. In this representation (called Gilbert's realization \cite{gilbert1963controllability}), the parameters are $D$ ($pm$ entries), $\lambda_i$'s ($N$ elements of the spectrum), and $K_i$ ($pm$ entries for each $i=1,...,N$). 
Thus we see that the number of parameters remains the same as the previous parameterization: $Npm + N + pm$ (see Fig. \ref{fig:gilbert_tf}). Nevertheless, this particular parameterization is often more descriptive of the actual system dynamics: not only are the input-output dynamics clearly defined, but as we will see in Section \ref{subsec:parSR}, each $\lambda_i$ is a mode of the system and the $D$ matrix represents how the inputs directly affect the outputs. 

\begin{example}
\label{ex:r1}
	Consider the representation in \eqref{eq:gilbert} where $K_i$ are all rank one matrices; this scenario arises when the poles of the system have algebraic multiplicity of one (i.e., they are distinct). For a particular $i$, instead of $K_i$ being parameterized by $pm$ parameters, here many of the parameters of $K_i$ are dependent on the other parameters, in order to ensure that the rank one constraint is met. Let $c_i \in \mathbb{R}^p$ and $b_i \in \mathbb{R}^m$ for all $i$, where $c_i b_i^T = K_i$. We can assume that the first entry of $c_i$ is one, if not, we could divide $c_i$ by its first entry and multiply $b_i$ by the same, thus leaving $K_i$ unchanged. Under this circumstance, each $K_i$ now only requires $p + m - 1$ entries to be specified in order to completely characterize it. Additionally, now any perturbations in these entries will uniquely change $K_i$. Therefore, these $p + m - 1$ entries are the parameters of $K_i$. Thus for this example, the total number of parameters for $G(s)$ is $Np + Nm + pm$ (see Fig. \ref{fig:ex1_tf}). 
\end{example}

We note that both parametrizations of the set of all $p \times m$ transfer functions of degree $N$, \eqref{eq:tfparam} and \eqref{eq:gilbert}, have the same number of parameters. Generically, almost any element of this set is uniquely identified by a choice of these parameters, and sufficiently rich input-output data enables the identification of the parameters specifying the system generating the data. As a result, in either case, we call these $Npm + N + pm$ parameters {\em identifiable}.  This definition of parameter identifiability is consistent with the definition of structurally identifiable \cite{bellman1970structural}.


\subsection{Parameterizations of State Realizations}\label{subsec:parSR}

In the previous section we saw that the set of all transfer functions of a given size and degree was efficiently parameterized by a particular number of identifiable parameters, and we demonstrated two distinct ways of doing this.  In this section we extend these ideas to state realizations.


Recall that the standard form for state realizations is given in $\eqref{eq:ss}$, and that we have assumed minimality. The set of all state realizations with $m$ inputs and $p$ outputs of order $n$ is parameterized by four matrices: $A \in \mathbb{R}^{n \times n}$, $B \in \mathbb{R}^{n \times m}$,
$C~\in~\mathbb{R}^{p \times n}$, and $D \in \mathbb{R}^{p \times m}$. Counting each entry in each matrix as a parameter yields a total of 
\begin{equation}\label{eq:ssparams}
n^2 + nm + np + pm
\end{equation} 
parameters for the set.

In the generic case, $A$ is diagonalizable with no repeated eigenvalues. This implies that there exists a transformation $T$ such that $x = T\tilde{x}$ and $T^{-1}AT = \tilde{A}$ is diagonal. Likewise, $T^{-1}B = \tilde{B}$, $CT = \tilde{C}$, and $D = \tilde{D}$. In this case, by \eqref{eq:ss2tf},
\begin{equation}
    \label{eq:stspsum}
    \begin{split}
        G(s) &= \tilde{C}(sI-\tilde{A})^{-1}\tilde{B} + \tilde{D} \\
        &= \tilde{D} + [\tilde{c}_1 \cdots \tilde{c}_n]
        \begin{bmatrix}
            \frac{1}{s-\lambda_1} & \cdots & 0 \\
            0 & \ddots & 0 \\
            0 & \cdots & \frac{1}{s-\lambda_n}
        \end{bmatrix}
        \begin{bmatrix}
            \tilde{b}_1^T \\
            \vdots \\
            \tilde{b}_n^T
        \end{bmatrix}\\
        &= \tilde{D} +
        \begin{bmatrix}
            \frac{\tilde{c}_1}{s-\lambda_1} & \cdots & \frac{\tilde{c}_n}{s-\lambda_n}
        \end{bmatrix}
        \begin{bmatrix}
            \tilde{b}_1^T \\
            \vdots \\
            \tilde{b}_n^T
        \end{bmatrix}\\
        &= \tilde{D} + \sum_{i=1}^n \frac{\tilde{c}_i\tilde{b}_i^T}{s-\lambda_i}. \\
    \end{split}
\end{equation}
Each term $\tilde{c}_i\tilde{b}_i^T$ is a rank one matrix (being the outer product of two vectors), thus it matches the transfer function parameterization given in \eqref{eq:gilbert}, satisfying the constraint given in Example \ref{ex:r1} that each $K_i$ is rank one. Example \ref{ex:r1} shows that there are $Nm + Np + pm$ parameters in the set of such transfer functions, where $N$ is the degree of $\delta(s)$. Since $A$ has no repeated eigenvalues, it follows from \eqref{eq:stspsum} that $N=n$. This implies that the set of corresponding transfer functions has $nm + np + pm$ parameters. Therefore, by \eqref{eq:ssparams}, the set of corresponding state realizations has $n^2$ more parameters.


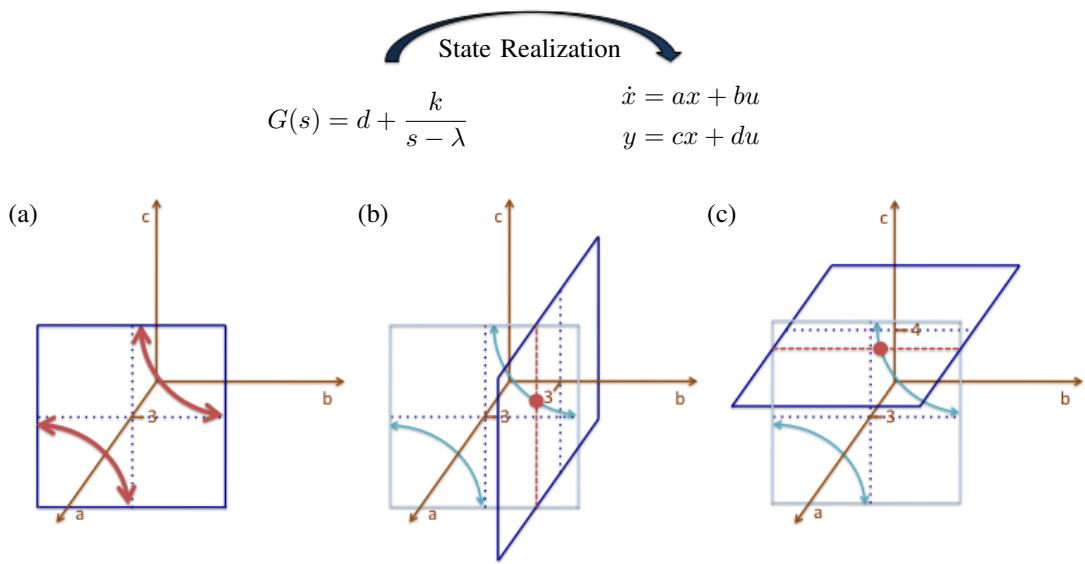
\begin{figure*}
\centering
\begin{overpic}[scale=0.75]{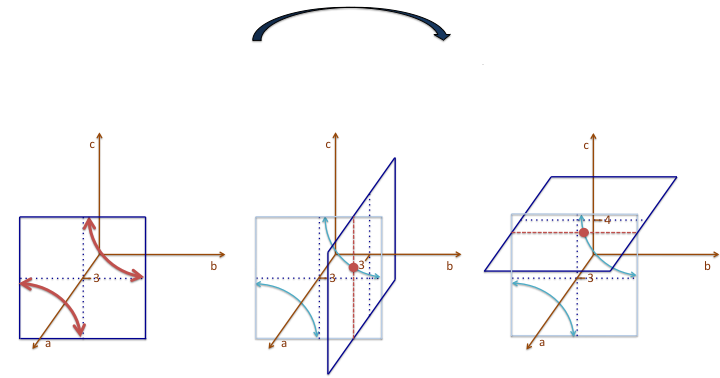}
	\put(40,47.5){{\parbox{\linewidth}{%
     State Realization}}}
     \put(-30,40){{\parbox{\linewidth}{%
     $$G(s) = d + \frac{k}{s - \lambda}$$}}}
     \put(0,41.5){{\parbox{\linewidth}{%
     \begin{align*}
     \dot{x} &= ax + bu \\
     y &= cx + du
     \end{align*}}}}
    \put(0,32){{\parbox{\linewidth}{%
     (a)}}}
     \put(32.5,32){{\parbox{\linewidth}{%
     (b)}}}
     \put(65,32){{\parbox{\linewidth}{%
     (c)}}}
\end{overpic}
\caption{A SISO first order transfer function has three identifiable parameters, while all minimal realizations of any such transfer function have four parameters, some of which are identifiable, conditionally identifiable, or structural.  Consider this example, $G(s)$, where $d=0$, $k=1$, and $\lambda = 3$.  All realizations of this transfer function have $a=3$ and $d=0$, so these two of the four state parameters are {\em identifiable}, leaving the other parameters, $b$ and $c$, to be {\em conditionally identifiable} (i.e. identifiable conditioned on a particular choice of other parameters).  Figure (a) shows the hyperbolic set of possible ($b$, $c$) combinations given $a=3$, and Figures (b) and (c) show how a different realization (a,b,c combination) is selected from the allowed set if, for example, one knew a priori that $b=3$ (Figure (b)) or $c=4$ (Figure (c)). In the case of fixing $b=3$ (or, respectively $c=4$), $b$ would be the {\em structural} parameter (known a priori, specifying a particular basis for the state space realization), enabling the remaining parameter $c$ (or $b$, respectively) to become identifiable \cite{Pare2014Thesis}.} 
\label{fig:philsfig}
\end{figure*}

This discrepancy in parameters is not surprising, since it is well-known that any transfer function will have infinitely many corresponding state realizations. Further, it is also well-known that any two state realizations of a transfer function are related by an $n \times n$ state transformation matrix $T$ as described prior to \eqref{eq:stspsum}. 
Therefore, once a $G$ has been specified with $nm + np + pm$ parameters, $n^2$ additional parameters are required to specify an instance of $(A, B, C, D)$. We say that these $n^2$ parameters are \emph{structural}, since they do not affect the system dynamics and are not identifiable from input-output data, but they do affect the internal structure of the system.

The matrices given in \eqref{eq:ss} are one way to parameterize the set of state realizations. In this parameterization, it is clear from \eqref{eq:gilbert} that $D$ can be identified directly from knowing $G$, so its entries are also identifiable parameters.  That is to say, every state space realization of $G$ has the same $D$ matrix.  However, the rest of the parameters cannot be identified simply by knowing $G$: one must also know $n^2$ additional parameters (fixing the structure, in a sense, by fixing the basis of the state space). For instance, if we fix $G$, there are many admissible choices for $B$ and $C$, but it could be the case that if $B$ were known, $C$ could be determined from data. It also is the case that the reverse is true: fixing $C$ would allow $B$ to be determined from data. Therefore, we refer to the entries of $B$ and $C$ as \emph{conditionally identifiable}: once a certain number of them have been fixed, the rest can be determined from data (see Fig. \ref{fig:standard_ss}). An illustrative example is also shown in Fig.~\ref{fig:philsfig}.


Conditionally identifiable parameters are not ideal when performing MBAM, because it is not clear whether or how the structure or dynamics will be affected. Therefore, it is important to find a state space parameterization which has only identifiable and structural parameters.



\subsection{Parameterization Using the Balanced Realization}
\label{subsec:bal}

When performing MBAM on a state realization, it is important to know which parameters will affect the dynamics of the system (identifiable) and which will affect the implementation of the dynamics (structural), since either may be fixed, depending on the application.  Parameterizations that partition the parameters into identifiable and structural without the need for conditionally identifiable parameters are useful when applying MBAM so that 
unintended parameters are not eliminated. Unique canonical realizations allow one to accomplish this goal. All parameters from a unique canonical realization can be identified from data, since the canonical realization fixes the structure. Then, any realization can be parameterized using the parameters of the canonical realization as the identifiable parameters and the transformation matrix $T$ between the two realizations as the structural parameters.

We illustrate this point by employing the balanced realization. Consider a minimal, stable system as in \eqref{eq:ss}. It is well known that there exists a state transformation from these matrices to an input--output equivalent balanced realization $(\bar{A},\bar{B},\bar{C},\bar{D})$ as in \eqref{eq:ssbal}, where $X$ is the diagonal matrix of HSVs ($\theta_i$'s) satisfying \eqref{eq:lap}. 
Notice the following simple statement is true for balanced realizations.
\begin{lemma} \label{lemma:rs}
Given a balanced realization $(\bar{A},\bar{B},\bar{C},\bar{D})$, $diag(\bar{B}\bar{B}^T) = diag(\bar{C}^T\bar{C})$.
\end{lemma}
\begin{proof}
Since $(\bar{A},\bar{B},\bar{C},\bar{D})$ is balanced, the observability and controllability Gramians are equal with the HSVs, $\theta_1,\dots,\theta_n$, on the diagonal. The diagonals of the Lyapunov equations give  
\begin{equation}\label{eq:diaglyp}
\begin{array}{lcl}
a_{ii}\theta_i + \theta_i a_{ii}  = -(\bar{C}^T\bar{C})_{ii}\\
a_{ii}\theta_i + \theta_i a_{ii}  = -(\bar{B}\bar{B}^T)_{ii}
\end{array}
\end{equation}
where the subscript $ii$ indicates the $i$th diagonal entry of the matrix and $a_{ii}$ is the $i$th diagonal entry of the $\bar{A}$ matrix. Therefore $diag(\bar{B}\bar{B}^T) = diag(\bar{C}^T\bar{C})$.
\end{proof}
\noindent Note that this lemma is also a result of Theorem 1 in \cite{smith2003generating}. 

Let the common diagonal entries of $\bar{B}$ and $\bar{C}$ be denoted by  $r_1^2, \dots r_n^2$. 
We can then write the $\bar{B}$ matrix as 
\begin{equation*}
\bar{B} = \begin{bmatrix}
r_1 \beta_1^T \\
\vdots \\
r_n \beta_n^T
\end{bmatrix}
\end{equation*}
where the $\beta_i$'s are a collection of normalized column vectors in $\RR^m$ satisfying $\beta_i^T \beta_i = 1$, for all $i = 1, \dots, n$.  Denoting the $n \times m$ matrix whose rows correspond to $\beta_i^T$ as $\beta^T$ and introducing $R = diag(r_i,\dots,r_n)$, it follows that $\bar{B} = R \beta^T$.  Clearly by construction $diag(\bar{B} \bar{B}^T) = (r_1^2,\dots,r_n^2)$.

Similarly, we can write $\bar{C}$ as
\begin{equation*}
\bar{C} = \begin{bmatrix}
r_1 \gamma_1 & \dots & r_n \gamma_n
\end{bmatrix}
\end{equation*}
where the $\gamma_i$'s are a collection of normalized column vectors in $\RR^p$ satisfying $\gamma_i^T \gamma_i = 1$, for all $i = 1, \dots, n$.  We write $\bar{C} = \gamma R$; note that $diag(\bar{C}^T \bar{C}) = (r_1^2, \dots, r_n^2) = diag(\bar{B} \bar{B}^T)$,  consistent with Lemma \ref{lemma:rs}.

Plugging $r_i^2$ into \eqref{eq:diaglyp} gives 
\begin{equation}\label{eq:aii}
\bar{a}_{ii} = -\frac{r_i^2}{2 \theta_i}
\end{equation} 
for the diagonal elements of the $\bar{A}$ matrix. From the off-diagonals of the Lyapunov equations we find, 
\begin{equation}\label{eq:aij}
\bar{a}_{ij} = r_i r_j \alpha_{ij},
\end{equation}
for $i \neq j$, where
\begin{equation*}
\alpha_{ij} = \frac{\theta_j (\beta^T \beta)_{ij} - \theta_i (\gamma^T \gamma)_{ij}  }{ \theta_i^2 - \theta_j^2}.
\end{equation*}

Leveraging these properties, we see that a balanced realization is specified by $(\theta, \beta, \gamma, R, D)$ as follows:
\begin{equation}
\begin{split}
\dot{\bar{x}}(t) & =  \bar{A}\bar{x}(t) + \begin{bmatrix}
r_1 \beta_1^T \\
\vdots \\
r_n \beta_n^T
\end{bmatrix}u(t) \\
y(t) & =  \begin{bmatrix}
r_1 \gamma_1 & \dots & r_n \gamma_n
\end{bmatrix}\bar{x}(t) + {D}u(t)
\end{split}
\label{eq:balparam}
\end{equation} 
where $\bar{A}$ is defined in \eqref{eq:aii} and \eqref{eq:aij}.  
We know that $\theta$ and $R$ both contain $n$ parameters and that $D$ contains $pm$ parameters. While $\beta$ is an $n \times m$ matrix, since $\beta^T_i\beta_i = 1$, specifying $n-1$ entries in $\beta_i$ fixes the magnitude of the final entry. Thus $\beta$ carries only $n(m-1)$ parameters. Likewise, $\gamma$ carries $n(p-1)$. It follows that any realization can be specified by these parameters along with an additional transformation matrix $T$ that dictates the change of basis from the balanced realization. 

Using this balanced parameterization, we again have $n^2 + nm + np + pm$ parameters (see Fig. \ref{fig:balanced_ss}), but now each parameter is clearly labeled as either identifiable or structural. 
Since any realization can be uniquely described by $T$, its transformation matrix to the balanced realization, we know that there are $n^2$ structural parameters. As discussed, the remaining $nm+np+pm$ parameters match the same number as $G(s)$ in Example $\ref{ex:r1}$, and therefore are all identifiable from data. As a result this parameterization enables one, when using MBAM to reduce the number of parameters, to be deliberate as to how the structure and/or dynamics are affected.

\begin{figure}
	\centering
    \begin{subfigure}[h]{0.45\textwidth}
    	\centering
        \includegraphics[scale=0.4]{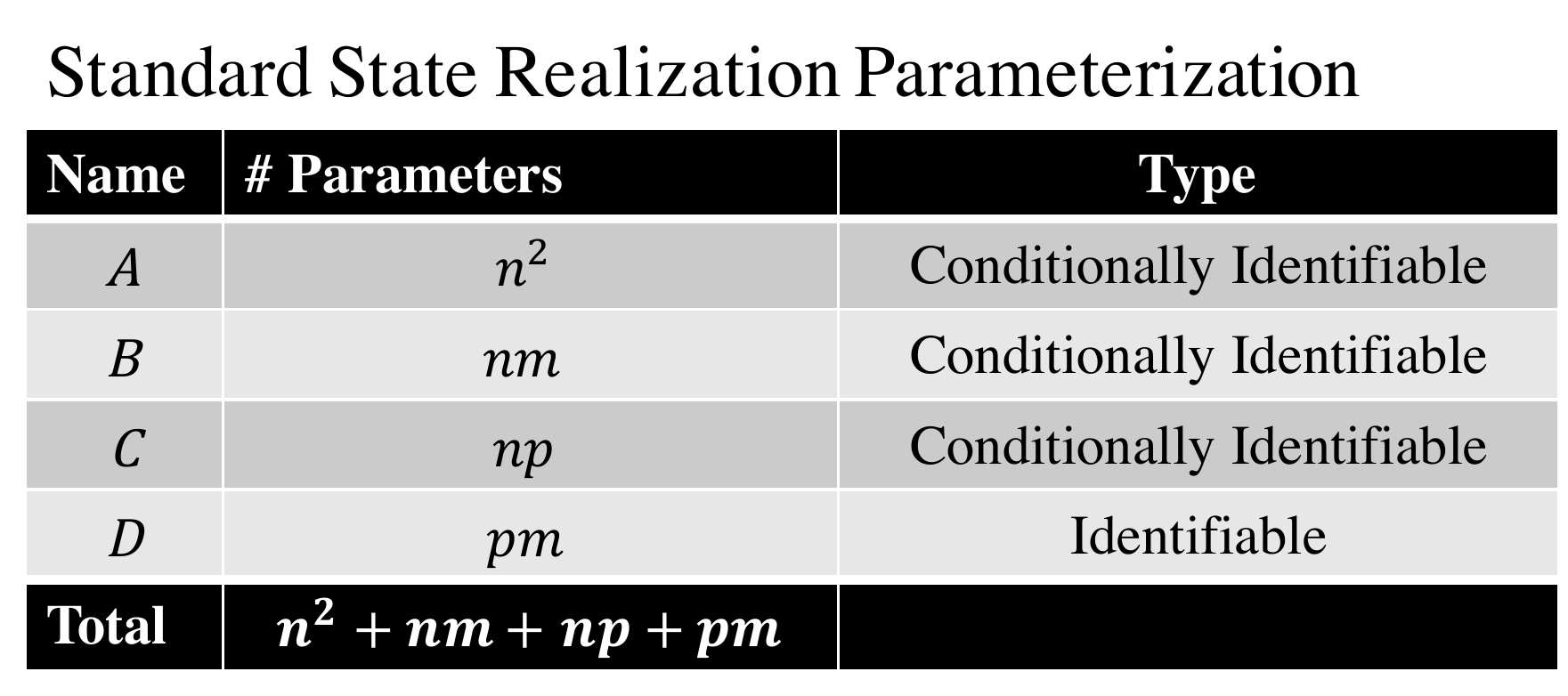}
        \caption{A minimal list of parameters that describe the set of all state space realizations with $m$ inputs and $p$ outputs of order $n$. This parameterization is given in \eqref{eq:ss}.}
        \label{fig:standard_ss}
    \end{subfigure}
    \begin{subfigure}[h]{0.45\textwidth}
    	\centering
        \includegraphics[scale=0.4]{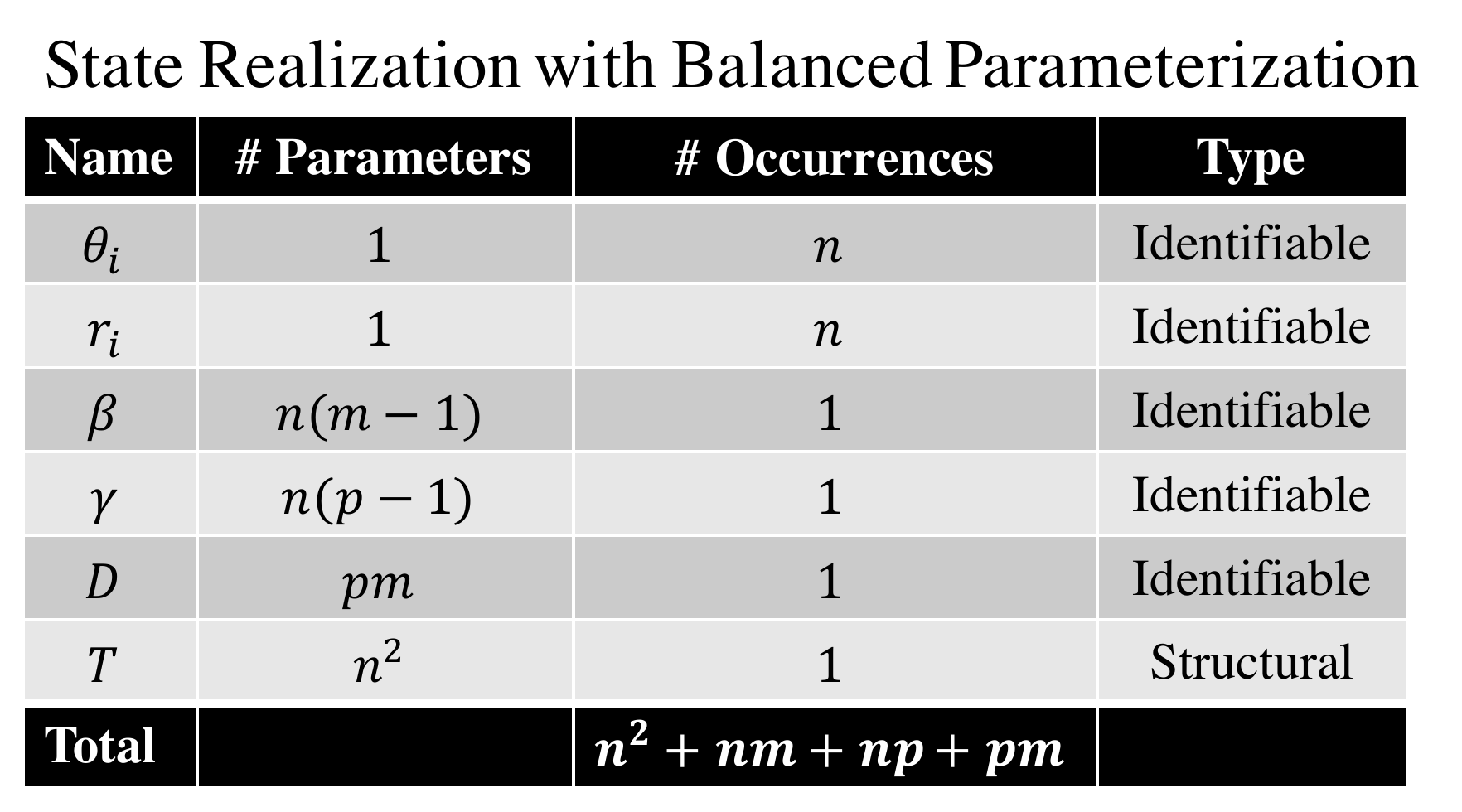}
        \caption{A minimal list of parameters that describe the state space realizations with $m$ inputs and $p$ outputs of order $n$. Note that the number of parameters is the same as in (a), but because any realization is described in terms of its relationship to the balanced realization, all parameters can be partitioned as identifiable or structural.}
        \label{fig:balanced_ss}
    \end{subfigure}
    \caption{Two ways to parameterize the set of state realizations.}
\end{figure}

\section{Unification of Model Reduction Techniques} \label{sec:main}

Given the parameterization of LTI systems of the previous section, we are equipped to apply the algorithm summarized in Figure~\ref{fig:algorithm}.  
In our experience solving the geodesic equation on multiple LTI systems, we empirically observed that the boundary approximations always take one of two forms when reducing the order of the system. 
This empirical observation suggests that it may not be necessary to sample the model predictions and solve Eq.~\ref{eq:geodesic}.
In other words, it is not necessary to explicitly construct the data space to find the mathematical form of the boundary approximation, all of this work can be done directly in parameter space.

In this section, we make this empirical observation rigorous through a sequence of theorems demonstrating that balanced truncation and balanced singular perturbation are each manifold boundary approximations of LTI systems.  
Note that in doing this, we have emphasized the conceptual distinction between a boundary approximation (see Definition~\ref{def:MBAM}) and the algorithm for finding them that requires solving a geodesic (see Figure~\ref{fig:algorithm}).



\subsection{Balanced Truncation from MBAM}

For the first theorem of this section we will restrict ourselves to considering the HSVs as the parameters, holding $R$, $\beta$, $\gamma$ and $D$ fixed.
\begin{theorem}\label{thm:bt}
Consider a balanced realization, as in \eqref{eq:balparam}, where the Hankel singular values ($\theta_i$'s) are parameters and the matrices $R$, $\beta$, $\gamma$ and $D$ are fixed. Balanced Truncation of $k$ states is equivalent to $k$ iterations of MBAM taking the relevant $\theta_i$'s $\rightarrow 0$.
\end{theorem}
\proof
Consider the equation for $\dot{x}_n(t)$ from \eqref{eq:balparam}:
\begin{equation}
\label{eq:xn}
\dot{\bar{x}}_n(t) =\sum^{n-1}_{i=1} r_i r_n \alpha_{ni} \bar{x}_i(t) 
-\frac{r_n^2 \bar{x}_n(t)}{2\theta_n} +  r_n \sum_{i=1}^m \beta_{ni} u_i(t) .
\end{equation}
Multiplying through by $\theta_n$ gives
\begin{equation*}
\theta_n \dot{\bar{x}}_n(t) =\theta_n \sum^{n-1}_{i=1} r_i r_n \alpha_{ni} \bar{x}_i(t) 
-\frac{r_n^2 \bar{x}_n(t)}{2} + \theta_n r_n \sum_{i=1}^m \beta_{ni} u_i(t).
\end{equation*}
Performing an MBAM approximation by taking the limit $\theta_n \rightarrow 0$ gives that $\bar{x}_n = 0$. Plugging this back into the dynamics of the rest of the system, i.e. $\dot{\bar{x}}_i(t),\ i<n$, gives
\begin{align}
\dot{\bar{x}}_i(t) &= -\frac{r_i^2 \bar{x}_i(t)}{2} + \sum_{j=1,j\neq i}^{n-1} r_i r_{j} \alpha_{ij} \bar{x}_j(t) +  r_i \sum_{j=1}^m \beta_{ij} u_j(t) , \nonumber 
\end{align}
which is equivalent to BT of one state. Clearly also having $x_n=0$ gives that 
$$y(t) 
= \begin{bmatrix}
r_1 \gamma_1 & \dots & r_{n-1} \gamma_{n-1}
\end{bmatrix} 
\begin{bmatrix}
\bar{x}_1(t) \\
\vdots \\
\bar{x}_{n-1}(t)
\end{bmatrix}.$$
Iterating this MBAM approximation $k-1$ more times, always choosing the smallest HSV, completes the proof.
\endproof

\begin{theorem}\label{thm:bt2}
Consider a balanced realization, as in \eqref{eq:balparam}, where  $r_1, \dots , r_n$ 
are parameters and the Hankel singular values, $\beta$, $\gamma$, and $D$ are fixed. Balanced Truncation of $k$ states is equivalent to $k$ iterations of MBAM taking the relevant  $r_i$'s~$\rightarrow 0$.
\end{theorem}
\proof
Consider the equation for $\dot{\bar{x}}_n(t)$ in \eqref{eq:xn}.
Performing an MBAM approximation by taking the limit $r_n \rightarrow 0$ gives that $\dot{\bar{x}}_n = 0$. 
 Plugging $r_n=0$ into the dynamics of the rest of the system, i.e. $\dot{\bar{x}}_i(t),\ i<n$, gives
\begin{align}
\dot{\bar{x}}_i(t) &= -\frac{r_i^2 \bar{x}_i(t)}{2} + \sum_{j=1,j\neq i}^{n-1} r_i r_{j} \alpha_{ij} \bar{x}_j(t) +  r_i \sum_{j=1}^m \beta_{ij} u_j(t). \nonumber 
\end{align}
Also, when $r_n = 0$,
$$y(t) =  \begin{bmatrix}
r_1 \gamma_1 & \dots & r_{n-1} \gamma_{n-1} & 0
\end{bmatrix} \bar{x}(t). $$
This gives BT of one state. Iterating this limit $k$ times, always choosing the $r_i$ with the largest subscript, completes the proof.
\endproof

Note that Theorems \ref{thm:bt} and \ref{thm:bt2} indicate two different paths along the model manifold that both converge to the same point on the boundary.  These paths are found by taking different limits of different parameters, but the common limit point is the Balanced Truncation approximation for the system.   


\subsection{Singular Perturbation Approximation from MBAM}

In a similar manner BSPA can be derived by applying MBAM to the balanced parameterization. 
\begin{theorem}\label{thm:spa}
Consider a balanced realization, as in \eqref{eq:balparam}, where  $r_1, \dots , r_n$ 
are parameters and the Hankel singular values, $\beta$, $\gamma$, and $D$ are fixed. Balanced Singular Perturbation Approximation of $k$ states is equivalent to $k$ iterations of MBAM taking the relevant  $r_i$'s~$\rightarrow \infty$.
\end{theorem}
%
\proof
We will prove this by induction, starting with the one state case.
Dividing \eqref{eq:xn} by $r_n^2$ gives
\begin{equation*}
\frac{1}{r_n^2}\dot{\bar{x}}_n(t) =\frac{1}{r_n} \sum^{n-1}_{i=1} r_i \alpha_{ni} \bar{x}_i(t) 
-\frac{ \bar{x}_n(t)}{2\theta_n} + \frac{1}{r_n} \sum_{i=1}^m \beta_{ni} u_i(t).
\end{equation*}
Letting $r_n \rightarrow \infty$ then yields 
\begin{equation*}
0 = -\frac{ \bar{x}_n(t)}{2\theta_n}
\end{equation*}
driving 
\begin{equation}\label{eq:xn0}
{\bar{x}}_n \rightarrow 0.
\end{equation}
Dividing \eqref{eq:xn} by $r_n$ gives
\begin{equation*}
\frac{1}{r_n} \left( \dot{\bar{x}}_n(t) \right) = \sum_{i=1}^{n-1} r_i \alpha_{ni} \bar{x}_i(t) 
-\frac{r_n \bar{x}_n(t)}{2 \theta_n} + \sum_{i=1}^m \beta_{ni} u_i(t).
\end{equation*}
Taking the limit as $r_n \rightarrow \infty$ with $r_n \bar{x}_n$ remaining finite, since, by \eqref{eq:xn0}, $\bar{x}_n \rightarrow 0$, 
gives
\begin{equation}\label{eq:rnxn}
r_n \bar{x}_n(t) = 2 \theta_n \sum_{i=1}^{n-1} r_i \alpha_{ni} \bar{x}_i(t) + 2 \theta_n \sum_{i=1}^m \beta_{ni} u_i(t).
\end{equation}
Furthermore, in the remaining equations for $\dot{\bar{x}}_i(t)$, $i~=~1, \dots, n-1$, and $y_i(t)$, $i = 1,\dots,p$, we find that $r_n$ and $\bar{x}_n$ always appear in the combination $r_n \bar{x}_n$.  Therefore, this limit is a well-defined boundary approximation for this parameterization.
Plugging \eqref{eq:rnxn} into the rest of the system, i.e. $\dot{\bar{x}}_i(t),\ i<n$, gives
\begin{align}
\dot{\bar{x}}_i(t) &= \sum_{j=1,j\neq i}^{n-1} (r_i r_{j} \alpha_{ij} + 2\theta_n r_i r_j \alpha_{in} \alpha_{nj}) \bar{x}_j(t) \nonumber \\
&\ \ \ \ \ \ \ \ +(2\theta_n r_i^2 \alpha_{in} \alpha_{ni} - \frac{r_{i}^2}{2\theta_i}) \bar{x}_{i}(t) \nonumber \\ 
&\ \ \ \ \ \ \ \ + \sum_{j=1}^m (r_i\beta_{ij}+ {2\theta_n}r_i\alpha_{in}\beta_{nj}) u_j(t) 
\label{eq:xi}\\
&= \sum_{j=1}^{n-1}( \bar{a}_{ij} - \frac{\bar{a}_{in}\bar{a}_{nj}}{\bar{a}_{nn}})\bar{x}_j(t) + \sum_{j=1}^m (\bar{b}_{ij} - \frac{\bar{a}_{in}}{\bar{a}_{nn}}\bar{b}_{nj}) u_j(t), \nonumber 
\end{align}
which is the system in \eqref{eq:spa} for $k=1$, that is, BSPA of one state. 
By similarly plugging \eqref{eq:rnxn} into $\bar{C}$ and $\bar{D}$, $\hat{C}$ and $\hat{D}$ from \eqref{eq:spa} for $k=1$ appear.


Now for the inductive step, assume BSPA of $k$ states is equivalent to $k$ MBAM approximations, giving the system in \eqref{eq:spa}, denoted by $(\hat{A},\hat{B},\hat{C},\hat{D})$. 
Let the system matrices that result from BSPA of $k+1$ states be denoted by $(\hat{A}^{k+1},\hat{B}^{k+1},\hat{C}^{k+1},\hat{D}^{k+1})$. 
Partition $\bar{A}$ such that 
\begin{equation*}
\bar{A} = \begin{bmatrix}
* & \bar{A}_{1,k+1} \\
\bar{A}_{k+1,1} & \bar{A}_{k+1,k+1}
\end{bmatrix}
= \begin{bmatrix}
\bar{*} & \bar{A}_{1,k} \\
\bar{A}_{k,1} & \bar{A}_{k,k}
\end{bmatrix}
\end{equation*}
where $\bar{A}_{k+1,k+1} \in \RR^{k+1 \times k+1}$ and $\bar{A}_{k,k} \in \RR^{k \times k}$; therefore, using the Schur complement notation, $\hat{A}^{k+1} = \bar{A}/\bar{A}_{k+1,k+1}$ and $\hat{A} = \bar{A}/\bar{A}_{k,k}$. 
By the Quotient Formula for Schur complements (Theorem 1.4 in \cite{zhang2006schur}), 
\begin{equation} \label{eq:schur}
\hat{A}^{k+1} = \bar{A}/\bar{A}_{k+1,k+1} =(\bar{A}/\bar{A}_{k,k})/(\bar{a}_{k+1,k+1}) 
\end{equation}
where 
\begin{equation}\label{eq:akkblock}
\bar{A}_{k+1,k+1} = 
\begin{bmatrix}
\bar{a}_{k+1,k+1} & \bar{a}^T_{k+1,k} \\ 
\bar{a}_{k,k+1} & \bar{A}_{k,k}
\end{bmatrix}
\end{equation}
with $\bar{A}_{k,k} \in \RR^{k \times k}$ and $\bar{a}_{k+1,k+1} \in \RR$. 
This means that $\hat{A}^{k+1}$ resulting from BSPA of $k+1$ states (the left hand side of \eqref{eq:schur}) is equivalent to the analogous part in \eqref{eq:xi} (the right hand side of \eqref{eq:schur}). 

Now we will show the same holds for $\hat{B}^{k+1}$, which is slightly more laborious since we cannot 
appeal to Schur complement properties. 
Partitioning $\bar{A}_{1,k}$ as
$$ \bar{A}_{1,k} = \begin{bmatrix} \tilde{A}_{1,k} \\ \bar{a}^T_{k+1,k} \end{bmatrix}$$
with $\bar{a}_{k+1,k}$ 
defined in \eqref{eq:akkblock} allows us to partition $\bar{A}_{1,k+1}$ as
\begin{equation*}
\bar{A}_{1,k+1} = \begin{bmatrix} \bar{a}_{k+1} & \tilde{A}_{1,k} \end{bmatrix}
\end{equation*}
where $\bar{a}_{k+1} \in \RR^{n-(k+1)}$. 
Using the above, \eqref{eq:spa}, and 
\eqref{eq:akkblock}, and partitioning $\bar{B}_{k+1}$ as 
%
%
\begin{align*}
\bar{B}_{k+1} &= \begin{bmatrix} \bar{b}_{k+1}^T \\ \bar{B}_{k} \end{bmatrix}
\end{align*}
with $\bar{B}_{k}\in \RR^{n-k \times m}$ and $\bar{b}_{k+1} \in \RR^m$, we have 
\begin{equation}\label{eq:ugly}
\hat{B}^{k+1} = \bar{B}_1 - \begin{bmatrix} \bar{a}_{k+1} & \tilde{A}_{1,k} \end{bmatrix} 
\bar{A}_{k+1,k+1}^{-1} 
\begin{bmatrix} \bar{b}_{k+1}^T \\ \bar{B}_{k} \end{bmatrix}
\end{equation}
where $\bar{B}_{1}\in \RR^{n-(k+1) \times m}$. 
By \eqref{eq:akkblock}, the block matrix inversion formula, matrix multiplication, and some rearranging of terms, 
\begin{multline*}
\begin{bmatrix} \bar{a}_{k+1} & \tilde{A}_{1,k} \end{bmatrix} 
\bar{A}_{k+1,k+1}^{-1} 
\begin{bmatrix} \bar{b}_{k+1}^T \\ \bar{B}_{k} \end{bmatrix} = \\ \bar{A}_{1,k}\bar{A}_{k,k}^{-1}\bar{B}_k + \frac{1}{c}(\bar{a}_{k+1}- \tilde{A}_{1,k}\bar{A}_{k,k}^{-1}\bar{a}_{k,k+1}) d
\end{multline*}
%
\noindent where 
\begin{align*}
c&=\bar{a}_{k+1,k+1}-\bar{a}^T_{k+1,k}\bar{A}_{k,k}^{-1}\bar{a}_{k,k+1} \text{ and}\\
d&=\bar{b}_{k+1}^T - \bar{a}^T_{k+1,k}\bar{A}_{k,k}^{-1}\bar{B}_k.
\end{align*} 
Partition $\hat{A}$ and $\hat{B}$ as 
\begin{equation*}
\hat{A} = \begin{bmatrix}
\hat{*} & \hat{a}_{1,n-k} \\
\hat{a}^T_{n-k,1} & \hat{a}_{n-k,n-k}
\end{bmatrix} \text{ and } 
\hat{B} = \begin{bmatrix} \hat{B}_1 \\ \hat{b}^T_{n-k} \end{bmatrix}
\end{equation*}
with $\hat{a}_{n-k,n-k} \in \RR$ and $\hat{b}_{n-k} \in \RR^{ m}$.
Note that $c$ is equal to $\hat{a}_{n-k,n-k}$, 
 $(\bar{a}_{k+1}-\tilde{A}_{1,k}\bar{A}_{k,k}^{-1}\bar{a}_{k,k+1})=\hat{a}_{1,n-k}$, 
and $d~=~\hat{b}^T_{n-k}$. 
Also note that $\bar{B}_1 - \bar{A}_{1,k}\bar{A}_{k,k}^{-1}\bar{B}_k=\hat{B}_1$. 
Therefore \eqref{eq:ugly} becomes
\begin{equation}\label{eq:pretty}
\hat{B}^{k+1} = \hat{B}_1 - \frac{1}{\hat{a}_{n-k,n-k}} \hat{a}_{1,n-k}\hat{b}^T_{n-k}.
\end{equation}
Similar expressions for  $\hat{C}^{k+1}$ and $\hat{D}^{k+1}$ can be found in an analogous way. Therefore \eqref{eq:schur}, \eqref{eq:pretty}, and the analogous expressions for $\hat{C}^{k+1}$ and $\hat{D}^{k+1}$ can be obtained, similar to \eqref{eq:schur}, by performing an iteration of MBAM on 
$(\hat{A},\hat{B},\hat{C},\hat{D})$ by taking $r_{n-k} \rightarrow~\infty$, thus completing the inductive step. 
Therefore, by induction, it holds for all $k$.
\endproof 
\noindent The resulting BSPA system is Hurwitz and is still a balanced realization with the $n-1$ HSVs, the same as the largest $n-1$ HSVs of the original system \cite{fernando1982singular}.

\textbf{Note:} In \eqref{eq:xi} if $\theta_n$, which is clearly no longer a HSV, is equal to its original value, then the reduced system is BSPA. However if $\theta_n$ is set to zero, then the resulting reduction is BT. Therefore, Theorem \ref{thm:spa} offers a whole new class of reduced systems ranging between BSPA and BT, and given a metric or goal, the optimal reduction can be found.
\section{Illustrative Example}\label{sec:example}

Consider the balanced parameterization for state realizations, presented in Section \ref{subsec:bal}, with $x \in \RR^2$,
\begin{align*}
\dot{\bar{x}}(t) & =  \begin{bmatrix}
-\frac{r_1^2}{2\theta_1} & \frac{r_1 r_2 (\theta_2-\theta_1))}{\theta_1^2-\theta_2^2}\\
\frac{r_2 r_1 (\theta_1-\theta_2))}{\theta_2^2-\theta_1^2} & -\frac{r_2^2}{2\theta_2}
\end{bmatrix}\bar{x}(t) +\begin{bmatrix}
r_1 \\ r_2 
\end{bmatrix}u(t) \\
y(t) & =  \begin{bmatrix}
r_1  & r_2
\end{bmatrix}\bar{x}(t) + du(t)
\end{align*}
where $\theta_1=r_1=1$ and $\theta_2$ and $r_2$ are the free parameters (while enforcing $\theta_2 \in [0,\theta_2]$), and we set $d=0$ and $\beta_1~=~\beta_2 = \gamma_1 = \gamma_2 = 1$. We can calculate the frequency responses of the different systems by varying the free parameters $\theta_2$ and $r_2$. See Fig. \ref{fig:ss2timeseries} for some sample frequency responses. We then use the frequencies $(s_1,s_2,s_3)$ to create the manifold in the data space.
The model manifold, in Figures \ref{fig:ss2modelmanifold}-\ref{fig:ss2modelmanifold3}, has two boundaries. The cyan boundary, at 
the top of the manifold, corresponds to the parameterizations where $r_2 \rightarrow \infty$  (we use one hundred since it is two orders of magnitude larger than one). 
The red boundary 
corresponds to parameterizations where $\theta_1=\theta_2$. The point at which the two boundaries meet at the bottom of the manifold, the magenta dot, is BT, corresponding to the parameterization where $\theta_2$ and/or $r_2 \rightarrow 0$. 
\begin{figure}
    \centering
    \begin{subfigure}[b]{0.5\textwidth}
      \begin{overpic}[width=\columnwidth]{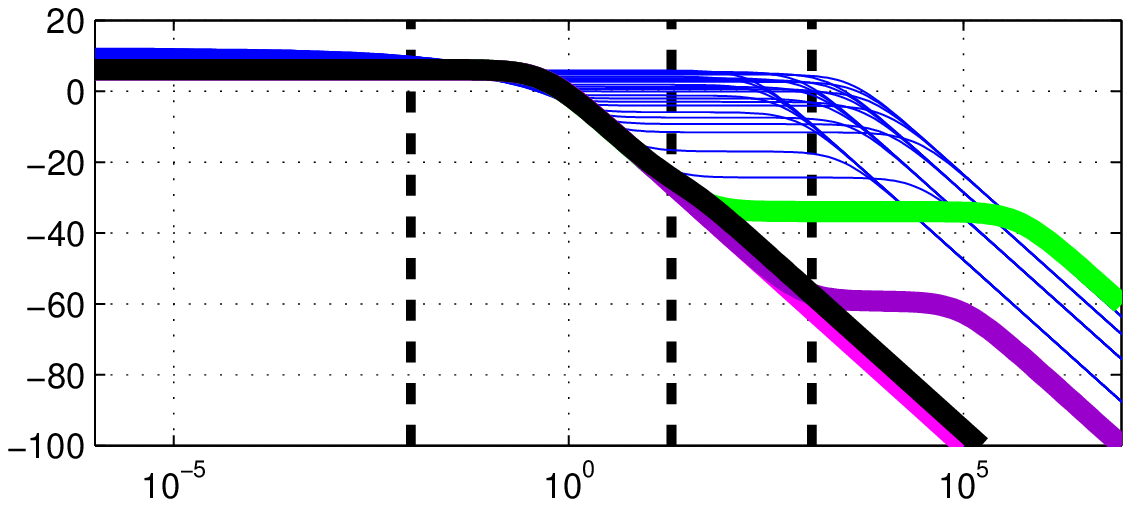}
      \put(39,-2.1){{\parbox{\linewidth}{%
     Frequency (rad/s)}}}
\end{overpic}
      \caption{Magnitude plot for the two--state space model.  Different curves are calculated by varying the parameters $\theta_2$ and $r_2$. The black line is the system with parameters $(\theta_2  , r_2) = ( 0.01, 0.8)$.}
      \label{fig:ss2timeseries}
    \end{subfigure}
    \begin{subfigure}[b]{0.495\textwidth}
    \begin{center}
      \begin{overpic}[width=\columnwidth]{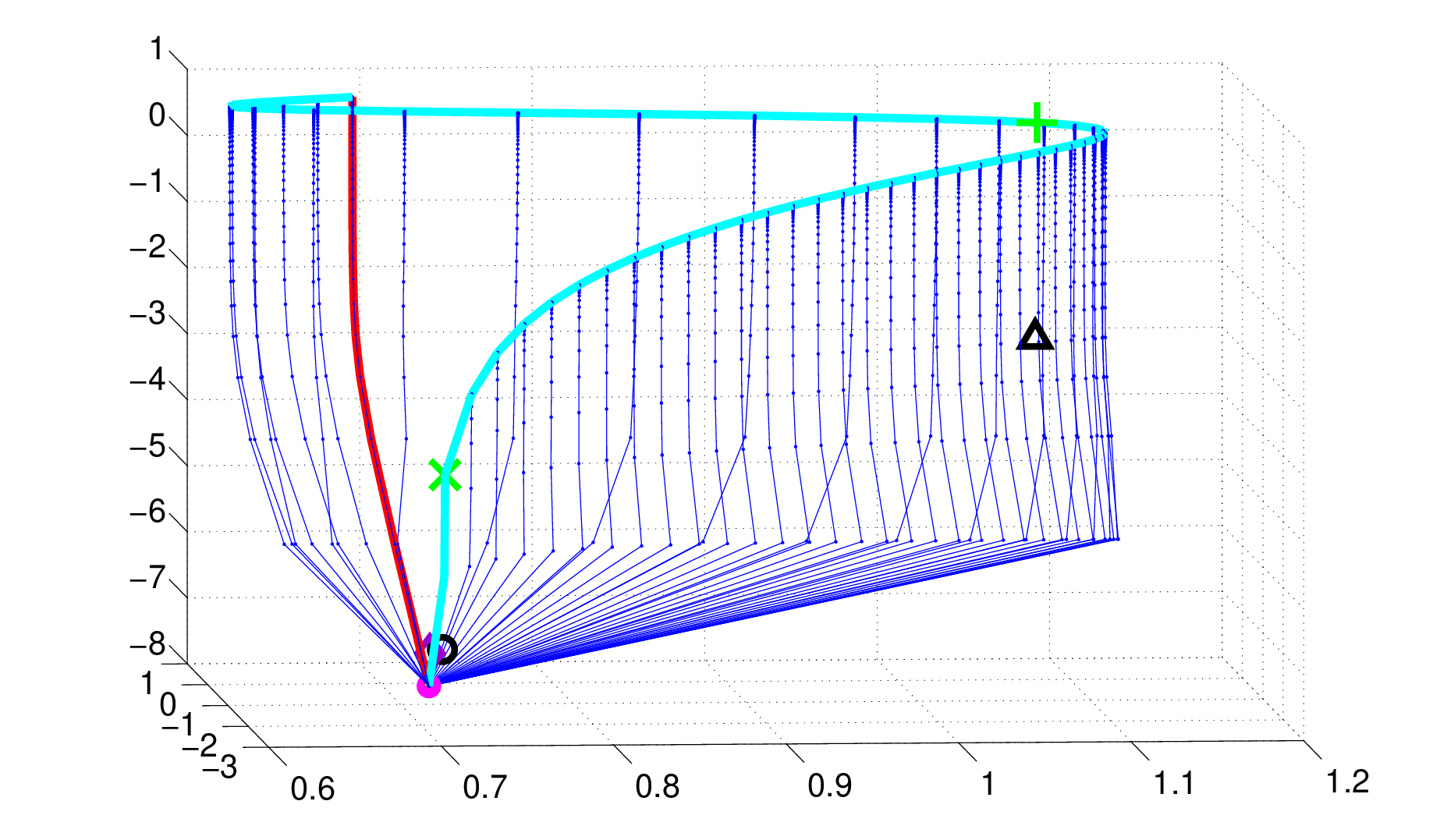}
\end{overpic}
    \end{center}
      \caption{The Model Manifold for the two--state space model:  Each point in the data space corresponds to a curve in Fig.~\ref{fig:ss2timeseries}. The cyan boundary indicates where $r_2 \rightarrow \infty$, the red boundary shows where $\theta_1=\theta_2$, and the magenta dot at the bottom is BT.}
      \label{fig:ss2modelmanifold}
    \end{subfigure}
    \begin{subfigure}[b]{0.495\textwidth}
     \begin{overpic}[width=\columnwidth]{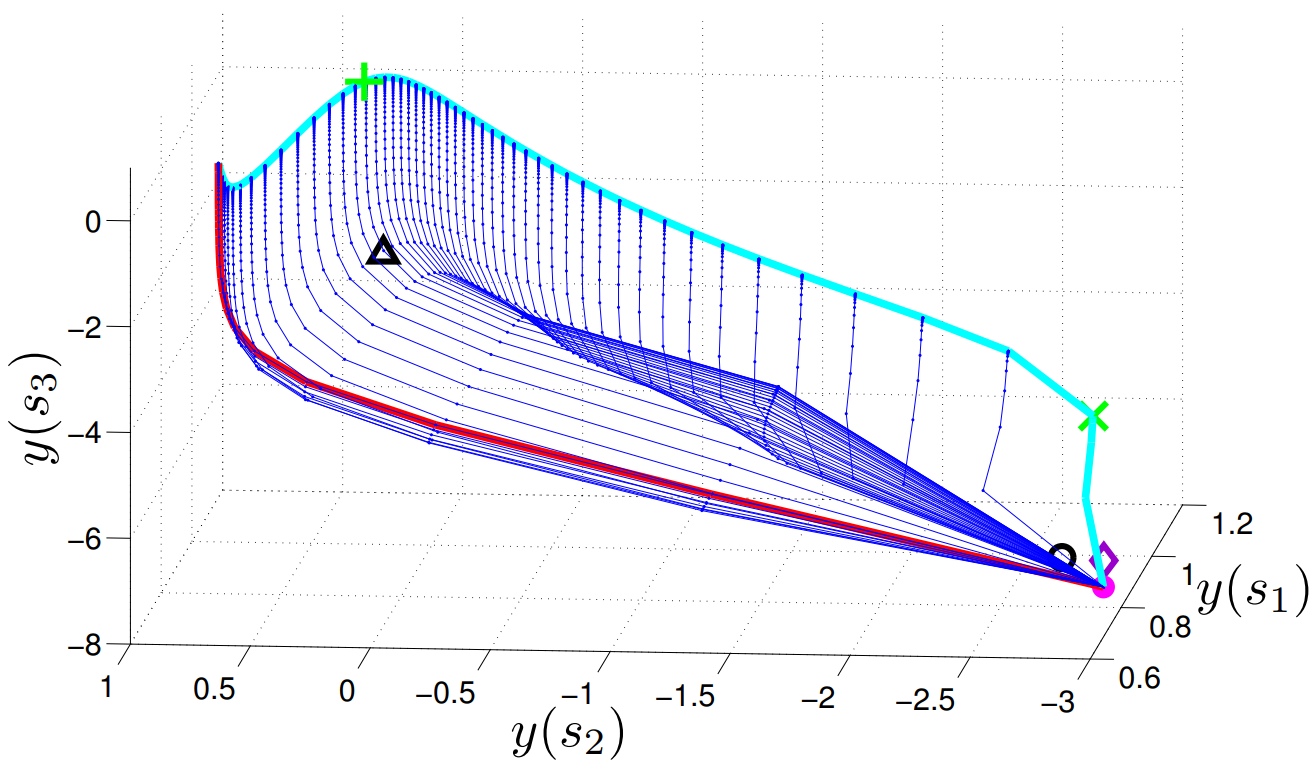}
\end{overpic}
      \caption{Second view of the Model Manifold for the two--state  model.}
      \label{fig:ss2modelmanifold2}
    \end{subfigure}
    \begin{subfigure}[b]{0.495\textwidth}
     \begin{overpic}[width=.7\columnwidth]{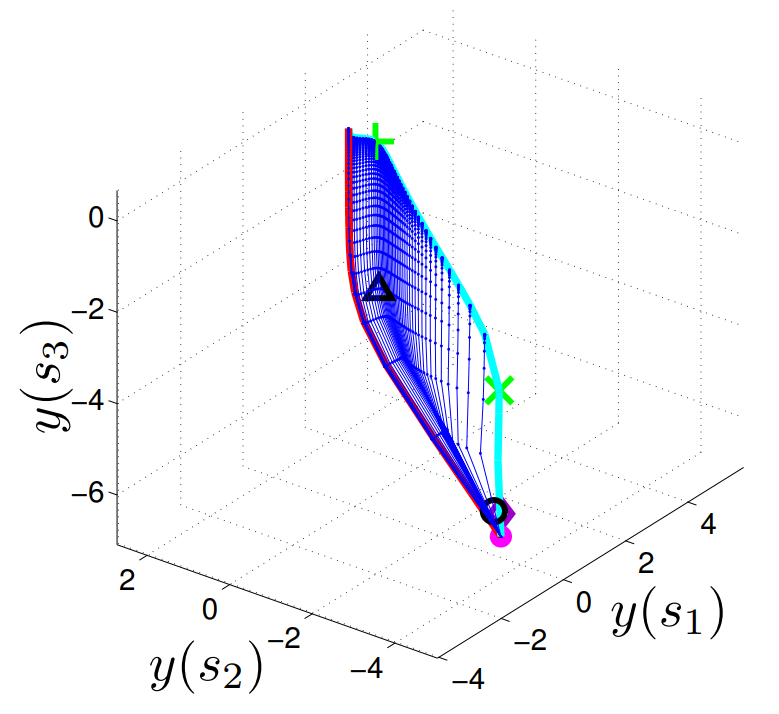}
\end{overpic}
      \caption{Third view with equal ratios on the axes to depict the thinness of the manifold.}
      \label{fig:ss2modelmanifold3}
    \end{subfigure}
    \caption{Magnitude plot and the resulting manifold for the two--state model.
    }
\label{fig:AER2}
\end{figure}


Consider the system given by the parameter values $(\theta_2  , r_2) = ( 0.7, 8)$, depicted on the manifold by the black triangle. It is clear that BT (the magenta dot) is not going to be a good approximation of the system by the distance between the two systems on the manifold. However, BSPA denoted by the green ``plus'' symbol is a very close approximation.   Consider another system given by the parameter values $(\theta_2  , r_2) = ( 0.01, 0.8)$, depicted on the manifold by the black circle. The reduced model given by BSPA, shown as a green ``x'', is still fairly close to the true system; however, it is clear that  BT (the magenta dot) is much closer. 

Although, as was explained in Section \ref{sec:back}, BT and BSPA have the same a priori $H_\infty$ error bound, BT typically gives better results at high frequencies while BSPA excels at low frequencies \cite{liu1989singular}.  The optimal approximation requires one to identify a metric customized to a context of interest. As a concrete example, returning to the system given by the parameter values $(\theta_2  , r_2) = ( 0.01, 0.8)$ (the green circle), the cyan curve connecting the black ``x'' to the magenta dot represents a family of candidate reduced models that interpolates between BT and BSPA.  The optimal reduced model relative to the three frequencies sampled, using the two-norm as the metric as measured in the data-space, is the point on this cyan curve that is closest to the green dot.  It is identified numerically as $(\theta_2  , r_2) = ( 5.3958 \times 10^{-4}, 0.8)$, depicted by the purple diamond.  Clearly this point is much closer to the original system than either BSPA or BT and therefore a better approximation for the chosen  metric of the two-norm.

Recall that the model manifold, in Figures \ref{fig:ss2modelmanifold}-\ref{fig:ss2modelmanifold3}, is constructed from the sampled frequencies $(s_1,s_2,s_3)$. Different samplings will give different manifolds, and therefore could provide better reduced models for different samplings. If a system designer had a certain frequency of importance $s^*$, a manifold could be constructed using sampled frequencies $(\hat{s}_1,\hat{s}_2,\hat{s}_3)$ around $s^*$, and find the best approximation of the system by finding the closest boundary point to the original system.



\section{Conclusion}\label{sec:con}


This paper  demonstrates how both Balanced Truncation and Singular Perturbation Approximations can be  viewed as a type of manifold boundary approximation.  The key idea unifying these different techniques is to choose a canonical parameterization for any given system that partitions its parameters into identifiable and structural sets, eliminating conditionally identifiable parameters.  Application of the Manifold Boundary Approximation Method, an information geometry  technique that requires neither linearity nor time invariance,  then  illustrates that each technique finds approximations on different boundaries of the system's model manifold embedded in an appropriate data space.  Depending on the choice of metric, this manifold can give insight into which approximation is the best for a given application and give alternative approximations interpolating between BT and BSPA.




As we saw in Section \ref{sec:mbam}, nothing about the MBAM approximation requires a linear model class, indeed, the Michaelis-Menten Reaction 
example was nonlinear. This fact, combined with the promising results for linear systems, suggests that by building an analogous parameterization for certain classes of nonlinear systems, similar types of approximations may be obtainable.   This framework, then, focuses attention on obtaining good parameterizations of nonlinear systems in order to recover approximations with desirable  qualities  similar to those of BT and  BSPA, questions that can be explored in future work. 
The work in \cite{scherpen1993balancing,scherpen1996balancing,lall2002subspace} explores extending balanced truncation to nonlinear systems, providing one possible reduced model. Alternatively, after the cost of constructing a parameterization, the MBAM approach, as illustrated in Section \ref{sec:example}, would provide a whole class of possible reduced models, extending BT and BSPA to  nonlinear systems and providing a spectrum of reduced-order systems interpolating between the two.



\bibliographystyle{IEEEtran}
\bibliography{IEEEabrv,bib}


\begin{IEEEbiography} [{\includegraphics[width=1in,height=1.25in,clip,keepaspectratio]{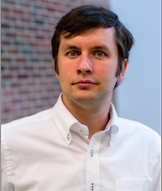}}]
{Philip E. Par\'{e}} received his B.S. in Mathematics with University Honors and his M.S. in Computer Science from Brigham Young University, Provo, UT, in 2012 and 2014, respectively, and his Ph.D. in Electrical and Computer Engineering from the University of Illinois at Urbana-Champaign, Urbana, IL in 2018.  He was the recipient of the 2017-2018 Robert T. Chien Memorial Award for excellence in research and named a 2017-2018 College of Engineering Mavis Future Faculty Fellow. 
His research interests include the modeling and control of dynamic networked systems, model reduction techniques, and time--varying systems.
\end{IEEEbiography}


\begin{IEEEbiography}[{\includegraphics[width=1in, keepaspectratio]{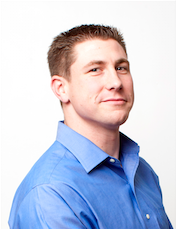}}]
{David Grimsman} is currently a Ph.D. candidate in the ECE Department at the University of California at Santa Barbara, where he is a recipient of the NSF Network Science IGERT Fellowship. He received a B.S. in Electrical Engineering from Brigham Young University in 2006 as a Heritage Scholar, and after some work in industry returned to academia to earn an M.S. in Computer Science from Brigham Young University in 2016. His research interests include distributed control, network systems, game theory, optimization, and algorithm design.
\end{IEEEbiography}


\begin{IEEEbiography}
[{\includegraphics[width=1in,clip,keepaspectratio]{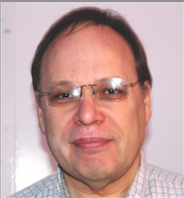}}]
{Alma Wilson} is a postdoctoral fellow in the CS Department at Brigham Young University. He received a B.Sc. in Physics and Mathematics from the University of Auckland, where he was a New Zealand University Junior Scholar and a Ngarimu V.C. Scholar, and he subsequently took a Ph.D. from Brigham Young University in 2000.  After work in industry, with stints in Arizona, Maryland, Australia, and Utah, he returned to academic life.  He is interested in the role of extensive properties in control theory, information theory, quantum theory, and thermodynamics, and he is particularly fond of constructing pedagogically advantageous representations of familiar STEM material.
\end{IEEEbiography}


\begin{IEEEbiography}
[{\includegraphics[width=1in,clip,keepaspectratio]{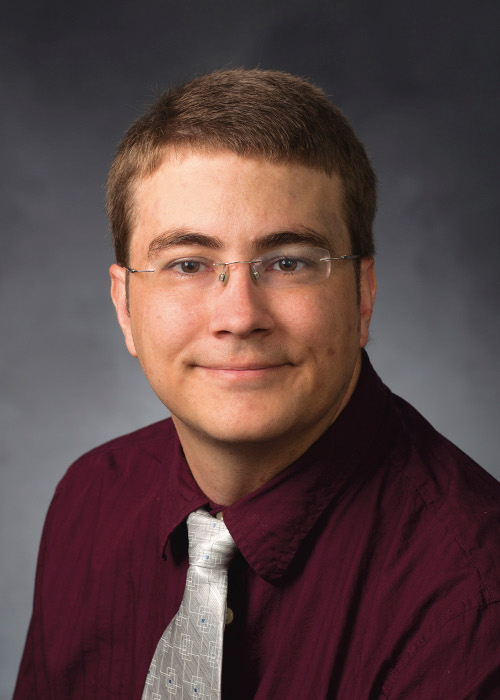}}]
{Mark K. Transtrum} received his M.S. and Ph.D. degrees in physics from Cornell University in 2010 and 2011 respectively.  Previously, he received his B.S. in Physics from Brigham Young University in 2006.  After graduate school, he studied computational biology as a Postdoctoral Fellow at MD Anderson Cancer Center.  Since 2013, he has been on the faculty of the Department of Physics \& Astronomy at Brigham Young University.  His work primarily uses information theory to study properties of mathematical models drawn from a variety of complex systems including power systems, systems biology, materials science, and neuroscience.    
\end{IEEEbiography}


\begin{IEEEbiography} [{\includegraphics[width=1in,clip,keepaspectratio]{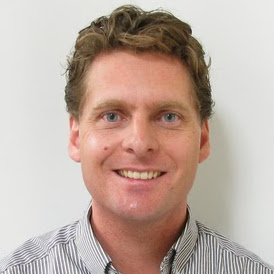}}]
{Sean Warnick} has been with the faculty of the Computer Science Department at Brigham Young University since 2003.  He received his Ph.D. and S.M. degrees from the Massachusetts Institute of Technology in 1995 and 2003, respectively, and his B.S.E. from Arizona State University in 1993. He attended ASU on scholarship from the Flinn Foundation, graduated summa cum laude with honors, and was named the Outstanding Graduate of the College of Engineering and Applied Sciences.  He has also held visiting positions from Cambridge University (2006), the University of Maryland at College Park (2008), and the University of Luxembourg's Centre for Systems Biomedicine (2014).  Sean was named the Distinguished Visiting Professor by the National Security Agency three years in a row, 2008-2010, for his work with their Summer Program for Operations Research Technology, and he has consulted with various companies.  Sean's  research interests  focus on  representations of complex networks of dynamic systems.
\end{IEEEbiography}

\end{document}